\documentclass[12pt]{amsart}
\usepackage[top=30truemm,bottom=30truemm,left=25truemm,right=25truemm]{geometry}
\usepackage{mathrsfs}
\usepackage{amsmath, amsthm, amssymb}
\usepackage{color}
\usepackage{bm}
\usepackage{amsfonts,amssymb}
\usepackage{dsfont}
\usepackage{amscd}
\usepackage{extarrows}
\usepackage{amsmath}
\usepackage{mathrsfs}
\usepackage{enumerate}
\usepackage{amscd}
\usepackage[all]{xy}
\usepackage[pagebackref,colorlinks]{hyperref}
\usepackage{geometry}
\newtheorem{thm}{Theorem}[section]
\newtheorem{defn}[thm]{Definition}
\newtheorem{lem}[thm]{Lemma}
\newtheorem{cor}[thm]{Corollary}
\newtheorem{prop}[thm]{Proposition}
\newtheorem{rem}[thm]{Remark}

\newcommand{\ind}[1]{\mathbf{1}_{#1}} 

\newcommand{\PSH}[2]{\text{PSH}(#1, #2)}
\newcommand{\env}[2]{P_{#2}(#1)}

\newcommand{\be}{\begin{equation}}
\newcommand{\ee}{\end{equation}}
\newcommand{\bea}{\begin{eqnarray}}
\newcommand{\eea}{\end{eqnarray}}
\newcommand{\ben}{\begin{eqnarray*}}
	\newcommand{\een}{\end{eqnarray*}}
\newcommand{\bt}{\begin{split}}
	\newcommand{\et}{\end{split}}
\newcommand{\bet}{\begin{equation}}

\begin{document}
\title[Weak convergence of Monge-Amp\`ere operator]
{Weak convergence of complex Monge-Amp\`ere operators on compact Hermitian manifolds}

\author[K. Pang]{Kai Pang}
\address{Kai Pang: School of Mathematical Sciences\\ Beijing Normal University\\ Beijing 100875\\ P. R. China}
\email{202331130031@mail.bnu.edu.cn}
\author[H. Sun]{Haoyuan Sun}
\address{Haoyuan Sun: School of Mathematical Sciences\\ Beijing Normal University\\ Beijing 100875\\ P. R. China}
\email{202321130022@mail.bnu.edu.cn}
\author[Z. Wang]{Zhiwei Wang}
\address{Zhiwei Wang: Laboratory of Mathematics and Complex Systems (Ministry of Education)\\ School of Mathematical Sciences\\ Beijing Normal University\\ Beijing 100875\\ P. R. China}
\email{zhiwei@bnu.edu.cn}

\begin{abstract}
     Let $(X,\omega)$ be a compact Hermitian manifold and let $\{\beta\}\in H^{1,1}(X,\mathbb R)$  be a real $(1,1)$-class with a smooth representative $\beta$, such that $\int_X\beta^n>0$. Assume that there is a  bounded $\beta$-plurisubharmonic function $\rho$ on $X$.  First, we provide a criterion for the weak convergence of non-pluripolar complex Monge-Amp\`ere measures associated to  a sequence of $\beta$-plurisubharmonic functions.  Second, this criterion is utilized to solve a degenerate complex Monge-Amp\`ere equation with an $L^1$-density. Finally, an $L^\infty$-estimate of the solution to the complex Monge-Amp\`ere equation for a finite positive Radon measure is given.
\end{abstract}


\maketitle

\tableofcontents

\section{Introduction}

 Let $(X, \omega)$ be a compact Hermitian manifold of complex dimension $n$, with a Hermitian metric $\omega$.   The nef cone and pseudoeffective cone in $H^{1,1}(X, \mathbb{R})$ are crucial objects of study, closely related to the geometric structure of the manifold $X$. However, the current research methods and tools for these cones are limited. Inspired by the significant work done in the K\"ahler setting (see e.g. \cite{DeP04,BDPP13}), it is believed that utilizing the Monge-Amp\`ere equation could be a promising approach.
   In the study of degenerate complex Monge-Amp\`ere equations, we usually need to consider the convergence of complex Monge-Amp\`ere operators.  For any Borel subset $E\subset X$, the capacity of $E$ with respect to $\omega$ is defined as:
   $$\text{Cap}_{\omega}(E):=\sup\left\{\int_E(\omega+dd^cv)^n:v\in \text{PSH}(X,\omega),0 \leq v \leq 1 \right\}.$$
   We say a sequence $\{\varphi_j\}_{j=1}^\infty\subset  \mbox{PSH}(X,\beta)$ converges in capacity to $\varphi\in \mbox{PSH}(X,\beta)$, if for a given $\varepsilon>0$,
   $$\lim_{j\rightarrow+\infty}\mbox{Cap}_{\omega}(|\varphi_j-\varphi|>\varepsilon)=0.$$
    It is known that the complex Monge-Amp\`ere operator is not continuous with respect to the $L^1$-topology (see an example due to Cegrell in \cite[Example 3.25]{GZ17}), and it is continuous with respect to the convergence in capacity of  uniformly bounded (quasi-)psh functions \cite{X09} on compact K\"ahler manifolds. When the sequence of quasi-psh functions is not uniformly bounded, there are also many important results for non-pluripolar complex Monge-Amp\`ere operators in the compact K\"ahler case; see \cite{P08,DV22,DP12,DDL23} and references therein.

  Recently, similar results were generalized to the Hermitian setting. Kolodziej-Nguyen \cite{KN22} proved the following interesting result: for a Hermitian metric $\omega$ on $X$, if the sequence $\{u_j\}$ is uniformly bounded $\omega$-psh functions such that $u_j\rightarrow u$ in $L^1(X,\omega^n)$ and $(\omega+dd^cu_j)^n\leq C(\omega+dd^c\varphi_j)$ for some uniformly bounded sequence $\{\varphi_j\} $ such that $\varphi_j\rightarrow \varphi\in PSH(X,\omega)$ in capacity, then $(\omega+dd^cu_j)^n$ converges weakly to $(\omega+dd^cu)^n$. This is a generalization of \cite[Lemma 2.1]{CK06} from the local setting to compact Hermitian manifolds that is pointed out in the proof of \cite[Lemma 2.11]{KN23b}  (see also \cite{KN23a}).

  The degenerate case was studied by Alehyane-Lu-Salouf \cite{ALS24}. Let $\beta$ be a smooth semipositive $(1,1)$-form with $\int_X\beta^n>0$, and assume that there is a Hermitian metric  $\omega$  on $X$ such that $dd^c\omega=0$ and $d\omega\wedge d^c\omega=0$. Let $u_j, u\in \mathcal E(X,\beta)$ and $u_j\rightarrow u$ in $L^1(X,\omega^n)$. It is proved in \cite{ALS24} that if $\langle(\beta+dd^cu_j)^n\rangle\leq \mu$ for some positive non-pluripolar measure, then $u_j\rightarrow u$ in capacity and the non-pluripolar complex Monge-Amp\`ere measure $\langle(\beta+dd^cu_j)^n\rangle\rightarrow \langle(\beta+dd^cu)^n\rangle$  in the weak sense. Here $\mathcal E(X,\beta)$ denotes the set of $\beta$-psh functions with full Monge-Amp\`ere mass, and $\langle(\beta+dd^c\cdot)^n\rangle$ represents the non-pluripolar complex Monge-Amp\`ere measure. For detailed definitions, see\cite [\S 2]{ALS24}.

 Now let $\beta$ be a smooth real closed $(1,1)$ form, which is not necessarily semipositive.
  A function $u:X\rightarrow [-\infty,+\infty)$ is called quasi-plurisubharmonic if locally $u$ can be written as the sum of a smooth function and a plurisubharmonic function.
  A $\beta$-plurisubharmonic ($\beta$-psh for short) function $u$ is defined as a quasi-plurisubharmonic function satisfying $\beta+dd^cu\geq 0$ in the sense of currents.
  The set of all $\beta$-psh functions on $X$ is denoted by $\mbox{PSH}(X,\beta)$.
  Suppose that there exists a function $\rho\in \mbox{PSH}(X,\beta)\cap L^{\infty}(X)$.
  We define a $(\beta+dd^c\rho)$-psh function to be a function $v$ such that $\rho+v\in \mbox{PSH}(X,\beta)$, and denote by $\mbox{PSH}(X,\beta+dd^c\rho)$ the set of all $(\beta+dd^c\rho)$-psh functions.

 Following \cite{BT87,GZ07,BEGZ10}, Li-Wang-Zhou introduced in \cite{LWZ24a} the non-pluripolar complex Monge-Amp\`ere operator  $\langle (\beta+dd^c\varphi)^n\rangle$ for  any   $\varphi\in \mbox{PSH}(X,\beta)$.
 Denote by $\mathcal E(X,\beta)$ the class of all $\varphi \in \text{PSH}(X,\beta)$ with full Monge-Amp\`ere mass, i.e.,
 $\int_X\langle (\beta+dd^c\varphi)^n\rangle=\int_X\beta^n$.

In this paper, we are interested in studying the criteria for the weak convergence of the non-pluripolar complex Monge-Amp\`ere operator. We establish the following theorem:

\begin{thm}\label{thm: main-1}
Let $(X,\omega)$ be a compact Hermitian manifold of complex dimension $n$, equipped with a Hermitian metric $\omega$. Let $\beta$ be a smooth closed real $(1,1)$-form with $\int_X\beta^n>0$.  Assume that there is a bounded $\beta$-psh function $\rho$ on $X$. Let  $u_j\in \mathcal{E}(X, \beta)$,  such that $(\beta +dd^c u_j)^n\leq \mu$ for some positive non-pluripolar Radon measure $\mu$. If $u_j\rightarrow u\in PSH(X, \beta)$ in $L^1(X)$, then $u\in \mathcal{E}(X, \beta)$ and $u_j\rightarrow u$ in capacity. Furthermore, $\langle(\beta+dd^cu_j)^n\rangle$   converges weakly to $\langle(\beta+dd^cu)^n\rangle$.
\end{thm}

\begin{rem}
The proof of this theorem is primarily inspired by \cite{ALS24}. In that work, the envelope plays a crucial role. Unlike the condition in \cite{ALS24}, where $\theta$ is assumed to be smooth and semi-positive, we only assume that the semi-positivity has a bounded weight function. This assumption is weaker both in terms of regularity and positivity, which introduces new challenges when analyzing the $(\beta + dd^c\rho)$-envelope. To overcome these difficulties, we introduce a new envelope (see Definition \ref{defn: new envelope}) and, building on the results from \cite{GLZ19} and \cite{ALS24}, we prove that this new envelope also possesses desirable properties. For instance, it preserves monotonicity, and its Monge-Amp\`ere mass is concentrated on the contact set  (see for example: Theorem \ref{thm: key}, Theorem \ref{thm: put_no_mass4}, Remark \ref{rem: second_version}).
\end{rem}

As an application of above theorem, we prove the following:
\begin{cor}
\label{cor: main}
Let $(X,\omega)$ be  a compact Hermitian manifold of complex dimension $n$, equipped with a Hermitian metric $\omega$. Let $\beta$ be a smooth closed real $(1,1)$-form with $\int_X\beta^n>0$.  Assume that there is a bounded $\beta$-psh function $\rho$ on $X$. Fix $\lambda > 0$, let $\mu$ be a positive Radon measure vanishing on pluripolar sets which can be writtten as $\mu=f \omega^n$, where $f\in L^{1}(\omega^n)$. Then there exist a unique $u \in \mathcal{E} (X, \beta)$ such that
$$ (\beta+dd^c u)^n = e^{\lambda u} \mu. $$
\end{cor}

\begin{rem}
If $f\in L^p(X,\omega)$ with $p>1$, the authors in \cite{LWZ24a} established the existence and uniqueness of bounded solutions to the above Monge-Amp\`ere equations. For more general results, we refer the reader to \cite{LWZ24b, LLZ24}.
\end{rem}

The $L^{\infty}$ estimate has a long history in the study of Monge-Amp\`ere equations. In the solution of the Calabi conjecture \cite{Yau78}, Yau introduced the Moser iteration, which is a key step in obtaining the $L^{\infty}$ estimate. This method works for equations with the right-hand side in $L^p$, where $p > n$ and $n$ is the dimension of the underlying compact K\"ahler manifold. 
Kolodziej \cite{Ko98} introduced a new method for the $L^{\infty}$ estimate using pluripotential techniques and integration by parts on compact K\"ahler manifolds. This method applies to $L^p$ densities, where $p > 1$. Kolodziej's approach has been further generalized to handle less positive or collapsing families of cohomology classes on K\"ahler manifolds; see, for example, \cite{EGZ09, EGZ08, DP10, BEGZ10}. There are also generalizations to compact Hermitian manifolds; see, for instance, \cite{LWZ24a} and the references therein.
Recently, Guo, Phong, and Tong \cite{GPT23} used a PDE approach to obtain a priori $L^{\infty}$ estimates for Monge-Amp\`ere equations.


In \cite{GL21, GL22, GL23}, Guedj and Lu introduced a new method to give $L^{\infty}$ estimates for some degenerate Monge-Amp\`ere equations on compact hermitian manifolds. Their method only relies on compactness and envelope properties of quasi-plurisubharmonic functions. We follow this approach to obtain a priori $L^\infty$-estimates for the solution to the degenerate complex Monge-Amp\`ere equation for a finite positive Radon measure.
\begin{thm}
\label{thm: linf est}
Let $(X,\omega)$ be  a compact Hermitian manifold of complex dimension $n$, equipped with a Hermitian metric $\omega$. Let $\beta$ be a smooth closed real $(1,1)$-form with $\int_X\beta^n>0$.  Assume that there is a bounded $\beta$-psh function $\rho$ on $X$.  Let $\mu$ be a finite positive Radon measure on X such that $PSH(X,\beta+dd^c\rho)\subset L^m(\mu)$ for some $m>n$ and $\mu(X)=\int_X\beta^n$. Then any solution $\varphi\in  PSH(X,\beta+dd^c\rho)\cap L^{\infty}(X)$ to $(\beta+dd^c(\rho+\varphi))^n=\mu$, satisfies
$$
Osc_X(\varphi)\leq T
$$
for some uniform constant T which only depends on $X,\beta$ and
$$
A_m(\mu):=sup\left\{(\int_X(-\psi)^md\mu)^{\frac{1}{m}}: \psi\in PSH(X,\beta_{\rho})\,with\,\underset{X}{sup}\,\psi=0\right\}.
$$
\end{thm}


\subsection*{Acknowledgements}  This research is supported by National Key R \& D Program of China (No. 2021YFA1002600)  and by NSFC grant (No. 12071035). The third author is partially supported by  the Fundamental Research Funds for the Central Universities.

\section{Preliminaries}

	Let $X$  be a compact Hermitian manifold of complex  dimension $n$ equipped  with a Hermitian metric $\omega$.
Let $\beta$ be a  closed smooth real $(1,1)$-form with $\int_X\beta^n>0$.
Suppose  there exists a function  $ \rho\in \mbox{PSH}(X,\beta)\cap L^{\infty}(X)$.
\subsection{Non-pluripolar product}
In this subsection, we collect several fundamental facts and properties concerning the non-pluripolar product. The non-pluripolar product is a crucial tool in the study of complex Monge-Amp\`ere equations and plays a significant role in understanding the behavior of plurisubharmonic (psh) functions and their associated measures.

\begin{lem}\label{pl_prop_}
Let $u, v \in PSH(X, \beta+dd^c \rho) \cap L^{\infty} (X)$, then  $$ \ind{\{\{u>v\}} (\beta+ dd^c\rho + dd^c u)^n = \ind{\{u>v\}} (\beta+ dd^c \rho + dd^c \mathrm{max} (u, v) )^n.$$
\end{lem}
\begin{proof}
Choose $\{\Omega_i \}_{i=1}^n$ an open cover of $X$ such that  there exist $\eta_i\in PSH(\Omega_i) \cap L^{\infty}(\Omega_i)$ satisfying $\beta+dd^c \rho = dd^c \eta_i$ on $\Omega_i$. On $\{u>v \}\cap \Omega_i = \{u+\eta_i > v + \eta_i\}\cap \Omega_i$, since $u+\eta_i = \max (u+\eta_i, v+\eta_i) $, by \cite[Corollary 4.2, Corollary 4.3]{BT87}, we have
$$ \ind{ \{u>v \}\cap \Omega_i} (dd^c(\eta_i + u))^n = \ind{ \{u>v \}\cap \Omega_i} (dd^c\max(\eta_i + u, \eta_i + v))^n.$$
Thus $$\ind{ \{u>v \}\cap \Omega_i} (\beta+dd^c \rho +dd^c u)^n = \ind{ \{u>v \}\cap \Omega_i} (\beta+dd^c \rho + dd^c\max(u, v))^n,$$ the lemma follows.
\end{proof}

\begin{cor}\label{pl_prop}
$u, v \in PSH(X, \beta) \cap L^{\infty}(X)$, then
$$ \ind{\{u>v\}} (\beta + dd^c u)^n = \ind{\{u>v\}} (\beta + dd^c \mathrm{max} (u, v) )^n.$$
\end{cor}
\begin{proof}
Since $u-\rho, v-\rho \in PSH(X, \beta+dd^c \rho)$, applying Lemma \ref{pl_prop_}, we get the result.
\end{proof}

\begin{lem}\label{plurifine locality_1}
Let $O$ be an plurifine open set, if $u,v \in PSH(X, \beta+dd^c \rho) \cap L^{\infty} (X)$ and $u=v$ on $O$, then
$$\ind{O} (\beta+dd^c \rho +dd^c u)^n = \ind{O} (\beta+dd^c \rho +dd^c v)^n. $$
\end{lem}
\begin{proof}
Follow Lemma \ref{pl_prop_}'s proof,  Choose $\{\Omega_i \}_{i=1}^n$ an open cover of $X$ such that  there exist $\eta_i\in PSH(\Omega_i)\cap L^{\infty}(\Omega_i)$ satisfying $\beta+dd^c \rho = dd^c \eta_i$ on $\Omega_i$. By \cite[Corollary 4.2, Corollary 4.3]{BT87}, we have $$\ind{O\cap \Omega_i } (dd^c(\eta_i + u))^n =\ind{O\cap \Omega_i } (dd^c(\eta_i + v))^n,$$
thus the lemma follows.
\end{proof}

\begin{cor}\label{cor: plurifine locality_2}
$u, v \in PSH(X, \beta) \cap L^{\infty}(X)$, $O$ is an plurifine open set, $u=v$ on $O$, then
$$ \ind{O} (\beta + dd^c u)^n = \ind{O} (\beta + dd^c v )^n.$$
\end{cor}
\begin{proof}
The same as the proof of Corollary \ref{pl_prop}.
\end{proof}

For $\varphi \in PSH(X, \beta)$, by Corollary \ref{pl_prop}, we have $\{ \ind{\{u>\rho-k\}}(\beta+dd^c \max (u, \rho-k))^n \}_{k\in \mathbb{N}}$ is a non-decreasing measure sequence.
From integration by part, we also have
\begin{align*}
&\sup_{k} \int_{K\cap\{ \varphi>\rho-k \}} \left[\beta+dd^c \max(\varphi, \rho-k) \right]^n\\
\leq  &\sup_{k} \int_X \left[\beta+dd^c \max(\varphi, \rho-k) \right]^n\\
=&\int_{X} \beta^n < +\infty,
\end{align*}
thus $\{ \ind{\{u>\rho-k\}}(\beta+dd^c \max (u, \rho-k))^n \}_{k\in \mathbb{N}}$ has convergence subsequence by Banach-Alaoglu theorem, and monotonicity implies that limit point is unique, thus the limit  $$ \beta_{u}^n:= \lim\limits_{k\rightarrow +\infty}  \ind{\{u>\rho-k\}}(\beta+dd^c \mathrm{max} (u, \rho-k))^n$$ is a well defined positive Radon measure.

\begin{defn}[{\cite{BEGZ10,LWZ24a}}]\label{defn: nn product}
	For any  $\varphi \in \mbox{PSH}(X,\beta)$, the non-pluripolar product $ \left\langle(\beta+dd^c \varphi)^n \right\rangle$ is defined as
	$$\langle \beta_\varphi^n \rangle :=\left\langle(\beta+dd^c \varphi)^n \right\rangle:=\lim_{k\rightarrow \infty} \ind{\{\varphi \textgreater \rho-k\}}(\beta+dd^c\varphi_k)^n,$$
	where $\varphi_k:=\max\{\varphi,\rho-k\}$.
\end{defn}
\begin{rem}
By the same argument as in \cite{BEGZ10}, the non-pluripolar product defined above possesses the  basic properties in \cite[Proposition 1.4]{BEGZ10}.
\end{rem}
\begin{thm} \label{thm: plur_prop1}
If $u, v \in PSH(X, \beta)$, then
$$ \ind{\{u>v\}} \langle(\beta+ dd^c u)^n\rangle = \ind{\{u>v\}} \langle(\beta+dd^c \mathrm{max} \{u, v \}   )^n\rangle.   $$
\end{thm}

\begin{proof}

For $j \in \mathbb{N}$, set
$ u_j = \max\{u, \rho - j\}, \quad v_j = \max\{v, \rho - j\}. $
From Lemma \ref{pl_prop}, we have
$$
\ind{\{u_j > v_j\}} \beta_{u_j}^n = \ind{\{u_j > v_j\}} \beta_{\max\{u_j, v_j\}}^n.
$$
Since $\{u_j > v_j\} = \{u > v\} \cap \{u > \rho - j\}$, we have
\begin{equation} \label{eq:11}
\ind{\{u > v\} \cap \{u > \rho - j\}} \beta_{u_j}^n = \ind{\{u > v\} \cap \{u > \rho - j\}} \beta_{\max\{u_j, v_j\}}^n.
\end{equation}

By the plurifine property (Lemma \ref{pl_prop}), for $k \geq j$,
\begin{align*}
\ind{\{u > \rho - j\}} \ind{\{u > \rho - k\}} \beta_{u_k}^n &= \ind{\{u > \rho - j\}} \beta_{u_k}^n \\
&= \ind{\{u_k > \rho - j\}} \beta_{u_k}^n \\
&= \ind{\{u_k > \rho - j\}} \beta_{\max\{u_k, \rho - j\}}^n \\
&= \ind{\{u > \rho - j\}} \beta_{u_j}^n.
\end{align*}
Thus, we have
$$
\ind{\{u > \rho - j\}} \ind{\{u > \rho - k\}} \beta_{u_k}^n = \ind{\{u > \rho - j\}} \beta_{u_j}^n.
$$
Letting $k \to +\infty$, by the definition of the non-pluripolar Monge-Amp\`ere measure, we get
\begin{equation} \label{eq:22}
\ind{\{u > \rho - j\}} \langle\beta_u^n\rangle = \ind{\{u > \rho - j\}} \beta_{u_j}^n.
\end{equation}

Similarly,
\begin{align*}
&\ind{\{u > v\} \cap \{u > \rho - j\}} \ind{\{\max(u, v) > \rho - k\}} (\beta + dd^c \max\{u, v\}_k)^n \\
&= \ind{\{u > v\} \cap \{u > \rho - j\}} \beta_{\max(u, v)_k}^n \\
&= \ind{\{u > v\} \cap \{\max(u, v)_k > \rho - j\}} \beta_{\max(u, v)_k}^n \\
&= \ind{\{u > v\} \cap \{\max(u, v)_k > \rho - j\}} \beta_{\max\{\max(u, v)_k, \rho - j\}}^n \\
&= \ind{\{u > v\} \cap \{u > \rho - j\}} \beta_{\max(u, v)_j}^n.
\end{align*}
Letting $k \to +\infty$, by the definition of the non-pluripolar Monge-Amp\`ere measure, we get
\begin{equation} \label{eq:33}
\ind{\{u > v\} \cap \{u > \rho - j\}} \langle\beta_{\max\{u, v\}}^n \rangle = \ind{\{u > v\} \cap \{u > \rho - j\}} \beta_{\max(u, v)_j}^n.
\end{equation}

Combining equations \eqref{eq:11}, \eqref{eq:22}, and \eqref{eq:33}, we obtain
$$
\ind{\{u > v\} \cap \{u > \rho - j\}} \langle\beta_u^n \rangle = \ind{\{u > v\} \cap \{u > \rho - j\}} \langle\beta _{\max\{u, v\}}^n\rangle.
$$
Letting $j \to +\infty$, and noting that $\langle\beta_u^n \rangle$ and $\langle\beta_{\max(u, v)}^n \rangle$ do not charge pluripolar sets, we conclude that
$$
\ind{\{u > v\}} \langle\beta_u^n \rangle = \ind{\{u > v\}} \langle\beta_{\max\{u, v\}}^n \rangle.
$$

\end{proof}

\begin{thm} \label{thm: plurifine_locality_3}
	Let $u, v \in \text{PSH}(X, \beta)$ and let $O$ be a plurifine open set such that $u = v$ on $O$. Then,
	$$
	\ind{O} \langle \beta_u^n \rangle = \ind{O} \langle \beta_v^n \rangle.
	$$
\end{thm}

\begin{proof}
	Following \cite[Proposition 1.4 (a)]{BEGZ10}, set $E_{k,j} := \{u > \rho - k\} \cap \{v > \rho - j\}$. By Corollary \ref{cor: plurifine locality_2},
	$$
	\ind{O \cap E_{k,j}} (\beta + dd^c u_k)^n = \ind{O \cap E_{k,j}} (\beta + dd^c v_j)^n.
	$$
	
	First, let $k \to +\infty$. Since $\{u > \rho - k\} \uparrow \{u > -\infty\}$, we have
	$$
	\ind{O \cap \{v > \rho - j\}} \langle (\beta + dd^c u)^n \rangle = \ind{O \cap \{v > \rho - j\} \cap \{u > -\infty\}} (\beta + dd^c v_j)^n.
	$$
	
	Next, let $j \to +\infty$. Since $\{v > \rho - j\} \uparrow \{v > -\infty\}$, we obtain
	$$
	\ind{O \cap \{v > -\infty\}} \langle (\beta + dd^c u)^n \rangle = \ind{O \cap \{u > -\infty\}} \langle(\beta + dd^c v)^n \rangle.
	$$
	
	Since the non-pluripolar Monge-Amp\`ere measures do not charge pluripolar sets, it follows that
	$$
	\ind{O} \langle \beta_u^n \rangle = \ind{O} \langle \beta_v^n \rangle.
	$$
\end{proof}

\begin{cor} \label{cor: comparison}
	If $u, v \in \text{PSH}(X, \beta)$, then
	$$
	\langle \beta_{\max(u,v)}^n \rangle  \geq \ind{\{u > v\}} \langle\beta_u^n \rangle + \ind{\{u \leq v\}} \langle\beta_v^n \rangle.
	$$
	In particular, if $u \leq v$ quasi-everywhere, then
	$$
	\ind{\{u = v\}} \langle \beta_v^n \rangle \geq \ind{\{u = v\}} \langle \beta_u^n \rangle.
	$$
\end{cor}

\begin{proof}
	The proof follows from \cite[Lemma 2.9]{DDL23}.
\end{proof}

\subsection{Quasi-plurisubharmonic envelopes}

If $h: X \to \mathbb{R} \cup \{\pm \infty\}$ is a measurable function, the $\beta$-plurisubharmonic (psh) envelope of $h$ is defined by
$$
\env{h}{\beta} = \left( \sup \left\{ \varphi \in \PSH{X}{\beta} : \varphi \leq h \quad \text{quasi-everywhere} \right\} \right)^*,
$$
with the convention that $\sup \emptyset = -\infty$. When $h = \min(u, v)$, we use the notation $\env{u, v}{\beta} = \env{\min(u, v)}{\beta}$.


\begin{lem} \label{lem: neg}
	Let $\{u_j\}_{j \in \mathbb{N}}$ be a sequence of $\beta$-psh functions on $X$ that is uniformly bounded above. Then the set $\{ (\sup_{j} u_j)^* > \sup_{j} u_j \}$ is a pluripolar set.
\end{lem}

\begin{proof}
	The result is local, so we may work on a coordinate ball $B$. On this ball, we can assume that $\beta = dd^c \psi$, where $\psi$ is smooth on $B$ and continuous on $\bar{B}$. We have:

	\begin{align*}
	\{ (\sup_{j} u_j)^* > \sup_{j} u_j \}
	&= \{ \psi + (\sup_{j} u_j)^* > \psi + \sup_{j} u_j \} \\
	&= \{ (\sup_{j} (u_j + \psi))^* > \sup_{j} (u_j + \psi) \}.
	\end{align*}

	The last set is negligible in the sense of \cite{BT82}, and by \cite[Theorem 7.1]{BT82}, it follows that the set above is pluripolar.
\end{proof}

\begin{lem} \label{lem: key1}
	If $h: X \to \mathbb{R} \cup \{-\infty\}$ is a measurable function that is bounded above, then either
	$$
	\sup \left\{ \varphi \in \PSH{X}{\beta} : \varphi \leq h \quad \text{quasi-everywhere} \right\}
	$$
	is bounded above on $X$, or  the set
	$$
	\left\{ \varphi \in \PSH{X}{\beta} : \varphi \leq h \quad \text{quasi-everywhere} \right\}
	$$
	is empty, in which case  by  convention,
	$$
	\sup \left\{ \varphi \in \PSH{X}{\beta} : \varphi \leq h \quad \text{quasi-everywhere} \right\} = -\infty.
	$$
\end{lem}

\begin{proof}
	First, if the set
	$$
	S := \left\{ \varphi \in \PSH{X}{\beta} : \varphi \leq h \quad \text{quasi-everywhere} \right\}
	$$
	is empty, the conclusion is immediate by definition. Therefore, we assume that $S$ is non-empty.
	
	Let $h: X \to \mathbb{R} \cup \{-\infty\}$ be a measurable function that is bounded above by some constant $C \in \mathbb{R}$. We need to show that every $\varphi \in S$ is bounded above by a uniform constant on $X$.
	
	Choose an open cover $\{U_i\}_{i=1}^n$ of $X$ such that on each $U_i$, $\beta = dd^c \phi_i$, where $\phi_i \in \mathcal{C}^\infty(U_i) \cap L^\infty(U_i)$. Let $M := \max_{i=1}^n \|\phi_i\|_\infty$.
	
	For any $\varphi \in S$ and any point $z_0 \in X$, there exists a ball $B_0$ centered at $z_0$ such that $B_0 \subseteq U_i$ for some $i$. Since $\phi_i + \varphi \in \text{PSH}(B_0)$, by the sub-mean value property of plurisubharmonic functions, we have:
	$$
	\phi_i(z_0) + \varphi(z_0) \leq \frac{1}{\text{Vol}(B_0)} \int_{B_0} (\phi_i(z) + \varphi(z)) \, dV.
	$$
	
	Since $\varphi \leq h \leq C$ quasi-everywhere on $X$, and $h$ is bounded above by $C$ outside a pluripolar set $P$, we can write:
	$$
	\phi_i(z_0) + \varphi(z_0) \leq \frac{1}{\text{Vol}(B_0)} \int_{B_0 \setminus P} (\phi_i(z) + \varphi(z)) \, dV.
	$$
	
	Given that $\|\phi_i\|_\infty \leq M$ and $\varphi \leq C$ on $B_0 \setminus P$, we have:
	$$
	\phi_i(z_0) + \varphi(z_0) \leq \frac{1}{\text{Vol}(B_0)} \int_{B_0 \setminus P} (C + M) \, dV = C + M.
	$$
	
	Thus,
	$$
	\varphi(z_0) \leq C + M - \phi_i(z_0).
	$$
	
	Since $\|\phi_i\|_\infty \leq M$, it follows that:
	$$
	\varphi(z_0) \leq C + M + M = C + 2M.
	$$
	
	Therefore, $\varphi$ is bounded above by $C + 2M$ for all $z_0 \in X$. This shows that the supremum of all such $\varphi$ is also bounded above by $C + 2M$.
	
	In conclusion, if $S$ is non-empty, then
	$$
	\sup \left\{ \varphi \in \PSH{X}{\beta} : \varphi \leq h \quad \text{quasi-everywhere} \right\}
	$$
	is bounded above by $C + 2M$ on $X$.
\end{proof}

We need the following lemma.
\begin{lem}[\cite{LLZ24}] \label{lem: pluripolar}
	If $\beta$ is a closed real $(1,1)$-form on $X$ with a bounded $\beta$-psh potential, then any locally pluripolar set in $X$ is globally pluripolar with respect to $\PSH{X}{\beta}$.
\end{lem}

\begin{prop} \label{prop: representation}
	Let $\beta$ be a closed real $(1,1)$-form on $X$.
	\begin{itemize}
	
\item [1.] If $h: X \to \mathbb{R}$ is a bounded measurable function, we define the $\beta$-psh envelope of $h$ as
	$$
	\tilde{P}_\beta(h)= \left( \sup \left\{ \varphi \in \PSH{X}{\beta} : \varphi \leq h \right\} \right)^*,
	$$
	then we have
	$$
	\tilde{P}_\beta(h)= \left( \sup \left\{ \varphi \in \PSH{X}{\beta} : \varphi \leq h \quad \text{quasi-everywhere} \right\} \right)^*=\env{h}{\beta}.
	$$
	In particular, if $h_j \downarrow h$ is a sequence of bounded functions, then
	$$
	\env{h_j}{\beta} \downarrow \env{h}{\beta}.
	$$
	
	\item[2.] If $h: X \to \mathbb{R} \cup \{-\infty\}$ is a measurable function that is bounded from above,
	if $h_j $ is a sequence of bounded functions and $h_j \downarrow h$, then
	$$
	\env{h_j}{\beta} \downarrow \env{h}{\beta}.
	$$
	\end{itemize}
\end{prop}

\begin{proof}
	\textbf{Proof of 1).}
	
	The proof follows  \cite[Proposition 2.2]{GLZ19}.
	
	Denote
	$$
	G := \sup \left\{ \varphi \in \PSH{X}{\beta} : \varphi \leq h \quad \text{quasi-everywhere} \right\}.
	$$
	If the set
	$$
	\left\{ \varphi \in \PSH{X}{\beta} : \varphi \leq h \quad \text{quasi-everywhere} \right\}
	$$
	is empty, then the set
	$$
	\left\{ \varphi \in \PSH{X}{\beta} : \varphi \leq h \right\}
	$$
	is also empty, and we have
	$$
	\tilde P_{\beta}(h) = \sup \left\{ \varphi \in \PSH{X}{\beta} : \varphi \leq h \quad \text{quasi-everywhere} \right\} = -\infty.
	$$
	
	Now assume $G \neq -\infty$. By Lemma \ref{lem: key1}, $G$ is bounded above. Using Choquet's Lemma, there exists a sequence $\{\varphi_j\} \subset \PSH{X}{\beta}$ such that $\varphi_j \leq h$ quasi-everywhere and
	$$
P_\beta(h)=	G^* = \left( \sup_j \varphi_j \right)^*.
	$$
	By Lemma \ref{lem: neg}, $G^* \in \PSH{X}{\beta}$ and $G^* \leq h$ quasi-everywhere. Therefore, by definition, $G^* = G$.
	
	We now show that
	$$
\tilde P_{\beta}({\max(h, G)})= G = G^*.
	$$
	On one hand, since $G = G^* \leq \max(h, G)$, it follows that
	$$
	G = G^* \leq \tilde P_{\beta}({\max(h, G)}).
	$$
	On the other hand, $\tilde P_{\beta}({\max(h, G)})\leq \max(h, G)$ quasi-everywhere. Since $\max(h, G) = h$ quasi-everywhere, we have
	$$
	\tilde P_{\beta}({\max(h, G)})\leq h \quad \text{quasi-everywhere},
	$$
	and thus
	$$
	\tilde P_{\beta}({\max(h, G)}) \leq G \quad \text{by definition}.
	$$
	Therefore,
	$$
\tilde P_{\beta}({\max(h, G)}) = G = G^*.
	$$
	
	Next, we prove that
	$$
	\tilde P_{\beta}({\max(h, G)})=\tilde P_{\beta}(h).
	$$
	Define
	$$
	E := \left\{ \tilde P_{\beta}({\max(h, G)}) \geq \max(h, G) \right\} \cup \left\{ \max(h, G) \neq h \right\},
	$$
	which is pluripolar. By Lemma \ref{lem: pluripolar}, there exists $\phi \in \PSH{X}{\beta}$ such that $E \subseteq \{\phi = -\infty\}$. For $\lambda \in (0, 1)$, we have
	$$
	\lambda \phi + (1-\lambda) \tilde P_{\beta}({\max(h, G)}) \in \PSH{X}{\beta}
	$$
	and is bounded above by $h$ (since $h$ is bounded). Thus, $\env{h}{\beta}$ is not identically $-\infty$. Letting $\lambda \to 0$, we get
	$$
	\tilde P_{\beta}({\max(h, G)})\leq \tilde P_{\beta}(h)\quad \text{quasi-everywhere}.
	$$
	Since both sides are $\beta$-psh, this inequality holds everywhere. The converse inequality is immediate. Therefore,
	$$
		\tilde P_{\beta}({\max(h, G)})=\tilde P_{\beta}(h).
	$$
	This completes the proof of the first part.
	
	\textbf{Proof of the decreasing property:}
	
	On one hand, by definition,
	$$
	\env{h_j}{\beta} \geq \env{h}{\beta}.
	$$
	On the other hand, $\env{h_j}{\beta} \leq h_j$ quasi-everywhere, so
	$$
	\lim_{j \to +\infty} \env{h_j}{\beta} \leq h \quad \text{quasi-everywhere}.
	$$
	If $\lim_{j \to +\infty} \env{h_j}{\beta} \in \PSH{X}{\beta}$, then by the first part of this theorem,
	$$
	\lim_{j \to +\infty} \env{h_j}{\beta} \leq \env{h}{\beta}.
	$$
	If $\lim_{j \to +\infty} \env{h_j}{\beta} = -\infty$, the conclusion is immediate.
	
	\textbf{Proof of 2).}
	
	The proof of the decreasing property for part 2) is identical to the second part of the proof of part 1).
\end{proof}

\begin{rem}\label{rem: equ of env}
	From above proof, we have already get that  for  measurable function $h: X \to \mathbb{R}\cup \{-\infty\}$ which is bounded from above,
$$P_\beta(h)= \sup \left\{ \varphi \in \PSH{X}{\beta} : \varphi \leq h \quad \text{quasi-everywhere} \right\} .$$
	\end{rem}

\begin{defn}
\label{defn: new envelope}
Assume $\rho$ is a bounded $\beta$-psh function. We define $u$ to be a $(\beta + dd^c \rho)$-psh function if $u + \rho$ is a $\beta$-psh function. Denote the set of all $(\beta + dd^c \rho)$-psh functions as $\PSH{X}{\beta + dd^c \rho}$.

If $\rho$ is a continuous $\beta$-psh function, then
$$
\PSH{X}{\beta + dd^c \rho} \subseteq \text{USC}(X),
$$
where $\text{USC}(X)$ denotes the set of all upper semicontinuous (u.s.c.) functions on $X$. For any function $h$, we have
$$
\left( \sup \left\{ \varphi \in \PSH{X}{\beta + dd^c \rho} : \varphi \leq h \right\} \right)^* = \left( \sup \left\{ \psi \in \PSH{X}{\beta} : \psi \leq h + \rho \right\} \right)^* - \rho.
$$

If $\rho$ is only assumed to be bounded, $(\beta + dd^c \rho)$-psh functions need not be upper semicontinuous (u.s.c.). Therefore, the upper semicontinuous regularization process is not suitable in this case.  Motivated by Remark \ref{rem: equ of env}, we define the $(\beta + dd^c \rho)$-psh envelope of $h$ as
$$
\env{h}{\beta + dd^c \rho} := \sup \left\{ \varphi \in \PSH{X}{\beta + dd^c \rho} : \varphi \leq h \quad \text{quasi-everywhere} \right\}.
$$
\end{defn}
Fortunately, one can prove that the envelope $P_{\beta + dd^c\rho}(h)$ shares similar properties as stated in Proposition \ref{prop: representation}.
\begin{thm} \label{thm: key}
	Let  $h$ be  a  measurable function that is bounded from above, then we have
	$$
	\env{h}{\beta + dd^c \rho} = \sup \left\{ \varphi \in \PSH{X}{\beta + dd^c \rho} : \varphi \leq h \quad \text{quasi-everywhere} \right\} = \env{h + \rho}{\beta} - \rho.
	$$
	
	In particular:
	\begin{enumerate}
		\item If $h_j \downarrow h$, where each $h_j$ is bounded from above, then
		$$
		\env{h_j}{\beta + dd^c \rho} \downarrow \env{h}{\beta + dd^c \rho}.
		$$
		
		\item If the set
		$$
		\left\{ \varphi \in \PSH{X}{\beta + dd^c \rho} : \varphi \leq h \right\}
		$$
		is non-empty, then $\env{h + \rho}{\beta} \in \PSH{X}{\beta}$, and thus
		$$
		\env{h}{\beta + dd^c \rho} \in \PSH{X}{\beta + dd^c \rho}.
		$$
	\end{enumerate}
\end{thm}
\begin{proof}
	By Remark \ref{rem: equ of env}, we have
	\begin{equation} \label{eq:key-proof}
	\begin{split}
	\env{h}{\beta + dd^c \rho}
	&= \sup \left\{ \varphi \in \PSH{X}{\beta + dd^c \rho} : \varphi \leq h \quad \text{quasi-everywhere} \right\} \\
	&= \sup \left\{ \psi \in \PSH{X}{\beta} : \psi \leq h + \rho \quad \text{quasi-everywhere} \right\} - \rho \\
	&= \env{h + \rho}{\beta} - \rho.
	\end{split}
	\end{equation}
	
\textbf{Proof of (1):}	If $h_j \downarrow h$, where each $h_j$ is bounded from above, then by part 2) of Proposition \ref{prop: representation},
	$$
	\env{h_j + \rho}{\beta} \downarrow \env{h + \rho}{\beta}.
	$$
	Therefore,
	$$
	\env{h_j}{\beta + dd^c \rho} = \env{h_j + \rho}{\beta} - \rho \downarrow \env{h + \rho}{\beta} - \rho = \env{h}{\beta + dd^c \rho}.
	$$

\textbf{Proof of (2):}	If the set
	$$
	\left\{ \varphi \in \PSH{X}{\beta + dd^c \rho} : \varphi \leq h \right\}
	$$
	is non-empty, then the set
	$$
	\left\{ \psi \in \PSH{X}{\beta} : \psi \leq h + \rho \right\}
	$$
	is also non-empty. By Proposition \ref{prop: representation}, we have $\env{h + \rho}{\beta} \in \PSH{X}{\beta}$, and thus
	$$
	\env{h}{\beta + dd^c \rho} = \env{h + \rho}{\beta} - \rho \in \PSH{X}{\beta + dd^c \rho}.
	$$
	
	The proof is completed.
\end{proof}

If $h$ is bounded, then by Choquet's lemma and the condition that $\rho$ is bounded, we have
$$
\env{h}{\beta} \in \PSH{X}{\beta} \cap L^{\infty}(X).
$$
Thus, $(\beta + dd^c \env{h}{\beta})^n$ is well-defined.

\begin{lem} \label{lem: put_no_mass1}
	If $h$ is a bounded Lebesgue measurable function, then $(\beta + dd^c \env{h}{\beta})^n$ puts no mass on $L(h) \cap \{ \env{h}{\beta} < h \}$, where $L(h)$ is the lower semi-continuity set of $h$.
\end{lem}

\begin{proof}
	The proof follows from Lemma 2.3 of \cite{GLZ19}.
	
	Since $h$ is bounded, from Proposition \ref{prop: representation}, we have $\env{h}{\beta}=\tilde P_\beta(h)$.
	Denote $\hat{h} := \env{h}{\beta}=\tilde P_\beta(h)$.
	
	 Fix $x_0 \in L(h) \cap \{ \hat{h} < h \}$. By the lower semi-continuity of $h$ and the upper semi-continuity of $\hat{h}$, there exists a ball $B$ such that $x_0 \in B \subseteq \{ \hat{h} < h \}$ and
\begin{align}\label{equ: ineq}
	\max_{\bar{B}} \hat{h} < \min_{\bar{B}} h - \delta
\end{align}
	for some $\delta > 0$.
	
	Since $\beta + dd^c \rho$ is a closed positive current, by the Dolbeault-Grothendieck Lemma, we can choose a Stein open neighborhood $D$ such that $\beta + dd^c \rho = dd^c \hat{\rho}$ on $D$, where $\hat{\rho} \in \text{PSH}({D}) \cap L_{\text{loc}}^1(D)$. 
Thus $\hat{\rho} - \rho$ is smooth and $\hat{\rho}$ is also bounded (shrink $D$ if necessary).
	
	Set $u := \hat{h} + (\hat{\rho} - \rho) \in \text{PSH}({D}) $, assuming $\operatorname{osc}_{\bar{B}} (\hat{\rho} - \rho) < \delta$.
	
	By \cite[Proposition 9.1]{BT82}, there exists a psh function $v$ on $D$ such that
\begin{align*}
	\left\{
	\begin{array}{ll}
	v = u & \text{on } D \setminus B, \\
	v \geq u & \text{on } D, \\
	(dd^c v)^n = 0 & \text{on } B.
	\end{array}
	\right.
\end{align*}
	
	On $\partial B$, we have
\begin{align*}
	v = u = \hat{h} + (\hat{\rho} - \rho) \leq \max_{\bar{B}} \hat{h} + \max_{\bar{B}} (\hat{\rho} - \rho).
\end{align*}
	By the maximum principle, it follows that
\begin{align*}
	v \leq \max_{\bar{B}} \hat{h} + \max_{\bar{B}} (\hat{\rho} - \rho) \quad \text{on } \bar{B}.
\end{align*}
	Thus, by inequality \eqref{equ: ineq},

	\begin{align*}
	v - (\hat{\rho} - \rho) &\leq \max_{\bar{B}} \hat{h} + \max_{\bar{B}} (\hat{\rho} - \rho) - (\hat{\rho} - \rho) \\
	&\leq \min_{\bar{B}} h - \delta + \operatorname{osc}_{\bar{B}} (\hat{\rho} - \rho) \\
	&\leq h \quad \text{on } B.
	\end{align*}

	Now define
\begin{align*}
	w :=
	\begin{cases}
	v - (\hat{\rho} - \rho) & \text{on } B, \\
	\hat{h} & \text{on } X \setminus B.
	\end{cases}
\end{align*}
	
	On the one hand, $w$ is $\beta$-psh on $X$, and $w \leq h$ on $X$. On the other hand, $w = v - (\hat{\rho} - \rho) \geq u - (\hat{\rho} - \rho) = \hat{h}$ on $B$. It follows that $w = \hat{h}$ on $B$, thus
\begin{align*}
	(\beta + dd^c \hat{h})^n = (dd^c v)^n = 0 \quad \text{on } B.
\end{align*}
	
	Therefore, $(\beta + dd^c \env{h}{\beta})^n$ puts no mass on $L(h) \cap \{ \env{h}{\beta} < h \}$.
\end{proof}

\begin{lem} \label{put_no_mass2}
	If $h$ is a quasi-continuous bounded Lebesgue measurable function, then $(\beta + dd^c \env{h}{\beta})^n$ puts no mass on $\{ \env{h}{\beta} < h \}$.
\end{lem}

\begin{proof}
	The proof follows the idea of \cite[Proposition 2.5]{GLZ19}.
	
	By definition, there exists a sequence of compact sets $(K_l)$ such that $\text{Cap}(X \setminus K_l) \leq 2^{-l}$ and the restriction $h|_{K_l}$ is a continuous function on $K_l$. By taking $\widetilde{K}_j := \cup_{1 \leq l \leq j} K_l$, we can assume the sequence $(K_j)$ is increasing.
	
	Using the Tietze-Urysohn Lemma, there exists a continuous function $H_j$ on $X$ such that $H_j|_{K_j} = h|_{K_j}$, and $H_j$ shares the same bounds as $h$.
	
	Set
	$$
	h_j := \sup \{ H_l \mid l \geq j \}.
	$$
	Then $\{ h_j \}_j$ is lower semicontinuous, quasi-continuous, and uniformly bounded:
	- For fixed $j$, and for $k \geq j$,
	$$
	h_j|_{K_k} = \max \left\{ \sup \{ H_l \mid l \geq k \}, \max (H_j, \dots, H_{k-1}) \right\} = \max \{ h, \max (H_j, \dots, H_{k-1}) \},
	$$
	which implies $h_j|_{K_k}$ is continuous. Since $j$ is arbitrary, $h_j$ is quasi-continuous.
	Since $h$ is bounded, by definition, $\{ h_j \}$ is uniformly bounded.
	
	By \cite[Lemma 2.4]{GLZ19}, $h_j$ converges to $h$ in capacity.
	
	By Hartogs' lemma, $\hat{h}_j \rightarrow \hat{h}$ in capacity. According to \cite[Theorem 2.6]{DDL23} (see also \cite[Theorem 2.6]{LLZ24}),
	$$
	\int_X h (\beta + dd^c \hat{h})^n \leq \liminf_{j \to +\infty} \int_X h_j (\beta + dd^c \hat{h}_j)^n.
	$$
	
	Since $\{ \hat{h}_j \}_j$ decreases to $\hat{h}$ by Proposition \ref{prop: representation}, we have
	$$
	\int_X \hat{h} (\beta + dd^c \hat{h})^n = \lim_{j \to +\infty} \int_X \hat{h}_j (\beta + dd^c \hat{h}_j)^n.
	$$
	
	Thus, by Lemma \ref{lem: neg} and the fact that Monge-Amp\`ere measures put no mass on pluripolar sets, we have
	$$
	\int_X (h - \hat{h}) (\beta + dd^c \hat{h})^n \leq \liminf_{j \to +\infty} \int_X (h_j - \hat{h}_j) (\beta + dd^c \hat{h}_j)^n.
	$$
	
	The right-hand term is equal to $0$ by Lemma \ref{lem: put_no_mass1}. Therefore,
	$$
	\int_X (h - \hat{h}) (\beta + dd^c \hat{h})^n = 0,
	$$
	which implies that $(\beta + dd^c \env{h}{\beta})^n$ puts no mass on $\{ \env{h}{\beta} < h \}$.
\end{proof}

\begin{rem} \label{rem: put_no_mass3}
	If $h$ is a quasi-continuous Lebesgue measurable function with a lower bound, and $\env{h}{\beta} \in \PSH{X}{\beta}$, then Lemma \ref{put_no_mass2} still holds. Actually, as in the proof of \cite[Theorem 2.3]{GL22}, one can replace $h$ by $\min(h, C)$, where $C > \sup_X \env{h}{\beta}$. Then we have
	$$
	\env{h}{\beta} = \env{\min(h, C)}{\beta},
	$$
	and the claim follows from Lemma \ref{put_no_mass2}.
\end{rem}

Now, we are ready to improve Lemma \ref{put_no_mass2} without the bounded condition, following the idea of \cite[Theorem 2.5]{ALS24}.

\begin{thm} \label{thm: put_no_mass4}
	Assume $h$ is quasi-continuous on $X$, and $\env{h}{\beta} \in \PSH{X}{\beta}$. Then $\env{h}{\beta} \leq h$ outside a pluripolar set, and we have
	$$
	\int_{\{ \env{h}{\beta} < h \}} \langle(\beta + dd^c \env{h}{\beta})^n\rangle = 0.
	$$
\end{thm}
\begin{proof}
	By Remark \ref{rem: put_no_mass3}, we may assume $h$ is bounded above. Set
	\begin{align*}
	h_j := \max(h, -j),
	\end{align*}
	which is a bounded function. Since $\rho$ is bounded, $\env{h_j}{\beta} \in \PSH{X}{\beta}$. By Lemma \ref{lem: neg}, $\env{h_j}{\beta} \leq h_j$ outside a pluripolar set. Since $\env{h}{\beta} \in \PSH{X}{\beta}$, also by Lemma \ref{lem: neg}, we conclude that $\env{h}{\beta} \leq h$ quasi-everywhere.
	
	Next, we prove the second result.
	
	By Lemma \ref{put_no_mass2} and Remark \ref{rem: put_no_mass3},
	\begin{align*}
	\int_{\{ \env{h_j}{\beta} < h_j \}} (\beta + dd^c \env{h_j}{\beta})^n = 0, \quad \forall j \geq 1.
	\end{align*}
	
	Fix $k \in \mathbb{N}$. For every $j \geq k$, we have $\{ \env{h_k}{\beta} < h \} \subseteq \{ \env{h_j}{\beta} < h_j \}$, which implies
	\begin{align*}
	\int_{\{ \env{h_k}{\beta} < h \}} (\beta + dd^c \env{h_j}{\beta})^n = 0.
	\end{align*}
	
	Fix $C > 0$. From Theorem \ref{thm: plur_prop1},
	\begin{align*}
	\int_{\{ \env{h_k}{\beta} < h \} \cap \{ \env{h}{\beta} > \rho - C \}} (\beta + dd^c \max(\env{h_j}{\beta}, \rho - C))^n = 0.
	\end{align*}
	
	Set
	\begin{align*}
	f_k := \left[ \max(h, \env{h_k}{\beta}) - \env{h_k}{\beta} \right] \times \left[ \max(\env{h}{\beta}, \rho - C) - (\rho - C) \right],
	\end{align*}
	which is a positive, bounded, quasi-continuous function. By \cite[Lemma 4.2]{KH09}, we have
	\begin{align*}
	\int_X f_k (\beta + dd^c \max(\env{h_j}{\beta}, \rho - C))^n = 0, \quad \forall j \geq k.
	\end{align*}
	
	By the Bedford-Taylor convergence theorem,
	\begin{align*}
	(\beta + dd^c \max(\env{h_j}{\beta}, \rho - C))^n \rightarrow (\beta + dd^c \max(\env{h}{\beta}, \rho - C))^n,
	\end{align*}
	which implies
	\begin{align*}
	\int_X f_k (\beta + dd^c \max(\env{h}{\beta}, \rho - C))^n &\leq \liminf_{j \to +\infty} \int_X f_k (\beta + dd^c \max(\env{h_j}{\beta}, \rho - C))^n \\
	&= 0,
	\end{align*}
	thanks to \cite[Theorem 2.6]{DDL23} (see also \cite[Theorem 2.6]{LLZ24}).
	
	The above is equivalent to
	\begin{align*}
	\int_{\{ \env{h_k}{\beta} < h \} \cap \{ \env{h}{\beta} > \rho - C \}} (\beta + dd^c \max(\env{h}{\beta}, \rho - C))^n = 0.
	\end{align*}
	By Theorem \ref{thm: plur_prop1},
	\begin{align*}
	\int_{\{ \env{h_k}{\beta} < h \} \cap \{ \env{h}{\beta} > \rho - C \}} \langle(\beta + dd^c \env{h}{\beta})^n\rangle = 0.
	\end{align*}
	
	Letting $k \to +\infty$ and then $C \to +\infty$, we obtain
	\begin{align*}
	\int_{\{ \env{h}{\beta} < h \}} \langle(\beta + dd^c \env{h}{\beta})^n \rangle= 0.
	\end{align*}
\end{proof}

\begin{cor} \label{cor: contactieq}
	Let $u, v \in \PSH{X}{\beta}$ be such that $\env{u, v}{\beta} \in \PSH{X}{\beta}$. Then
	\begin{align*}
	\langle\beta_{\env{u, v}{\beta}}^n \rangle &\leq \ind{\{\env{u, v}{\beta} = u\}} \langle\beta_u^n \rangle + \ind{\{\env{u, v}{\beta} = v, \env{u, v}{\beta} < u\}} \langle \beta_v^n \rangle.
	\end{align*}
	In particular, if $\mu$ is a positive measure such that $\langle\beta_u^n \rangle \leq \mu$ and $\langle \beta_v^n \rangle \leq \mu$, then $\langle \beta_{\env{u, v}{\beta}}^n \rangle \leq \mu$.
\end{cor}

\begin{proof}
	By Theorem \ref{thm: put_no_mass4}, $\langle\beta_{\env{u, v}{\beta}}^n\rangle$ is supported on the contact set $ K := K_u \cup K_v $, where
	\begin{align*}
	K_u &:= \{ \env{u, v}{\beta} = u \}, \\
	K_v &:= \{ \env{u, v}{\beta} = v \} \cap \{ \env{u, v}{\beta} < u \}.
	\end{align*}
	
	From Corollary \ref{cor: comparison}, we have
	\begin{align*}
	\ind{K_u} \langle \beta_{\env{u, v}{\beta}}^n \rangle &\leq \ind{K_u} \langle \beta_u^n \rangle, \\
	\ind{K_v} \langle \beta_{\env{u, v}{\beta}}^n \rangle&\leq \ind{K_v} \langle\beta_v^n \rangle.
	\end{align*}
	
	Adding these two inequalities, we obtain
	\begin{align*}
	\langle\beta_{\env{u, v}{\beta}}^n\rangle &= \ind{K_u} \langle\beta_{\env{u, v}{\beta}}^n \rangle + \ind{K_v} \langle \beta_{\env{u, v}{\beta}}^n \rangle\\
	&\leq \ind{K_u} \langle\beta_u^n \rangle + \ind{K_v}\langle \beta_v^n \rangle,
	\end{align*}
	which completes the proof.
	
	For the second part, if $\mu$ is a positive measure such that $\beta_u^n \leq \mu$ and $\beta_v^n \leq \mu$, then by the above inequality,
	\begin{align*}
	\langle\beta_{\env{u, v}{\beta}}^n \rangle &\leq \ind{K_u} \langle \beta_u^n \rangle + \ind{K_v} \langle \beta_v^n \rangle \\
	&\leq \ind{K_u} \mu + \ind{K_v} \mu \\
	&= \mu.
	\end{align*}
\end{proof}

\begin{rem}
	The corollary above is true if we replace $\beta$ by $\beta + dd^c \rho$ and $\PSH{X}{\beta}$ by $\PSH{X}{\beta + dd^c \rho}$ due to the equality in Theorem \ref{thm: key}.
\end{rem}

\subsection{The full mass class}


\begin{defn}
	The full mass class $\mathcal{E}(X, \beta)$ is defined as follows:
	\begin{align*}
	\mathcal{E}(X, \beta) &:= \left\{ u \in \PSH{X}{\beta} \mid \lim_{t \to +\infty} (\beta + dd^c \max(u, \rho - t))^n(\{u \leq \rho - t\}) = 0 \right\} \\
	&= \left\{ u \in \PSH{X}{\beta} \mid \int_X \langle(\beta + dd^c u)^n\rangle = \int_X \beta^n \right\}.
	\end{align*}
\end{defn}

The second equation holds because:
\begin{align*}
\int_X \langle(\beta + dd^c u)^n \rangle &= \lim_{t \to +\infty} \int_X \ind{\{u > \rho - t\}} (\beta + dd^c \max(u, \rho - t))^n \\
&= \lim_{t \to +\infty} \left[ \int_X (\beta + dd^c \max(u, \rho - t))^n - \int_X \ind{\{u \leq \rho - t\}} (\beta + dd^c \max(u, \rho - t))^n \right] \\
&= \int_X \beta^n - \lim_{t \to +\infty} (\beta + dd^c \max(u, \rho - t))^n(\{u \leq \rho - t\}).
\end{align*}

Similarly, we define the full mass class for $\beta + dd^c \rho$:

\begin{defn} \label{class2}
	The full mass class $\mathcal{E}(X, \beta + dd^c \rho)$ is defined as follows:
	\begin{align*}
	\mathcal{E}(X, \beta + dd^c \rho) &:= \left\{ u \in \PSH{X}{\beta + dd^c \rho} \mid \lim_{t \to +\infty} (\beta + dd^c \rho + dd^c \max(u, -t))^n(\{u \leq -t\}) = 0 \right\} \\
	&= \left\{ u \in \PSH{X}{\beta + dd^c \rho} \mid \lim_{t \to +\infty} (\beta + dd^c \max(\rho + u, \rho - t))^n(\{u + \rho \leq \rho - t\}) = 0 \right\} \\
	&= \left\{ u \in \PSH{X}{\beta + dd^c \rho} \mid u + \rho \in \mathcal{E}(X, \beta) \right\} \\
	&= \left\{ u \in \PSH{X}{\beta + dd^c \rho} \mid \int_X \langle (\beta + dd^c \rho + dd^c u)^n \rangle= \int_X \beta^n \right\}.
	\end{align*}
\end{defn}

\begin{rem} \label{rem: second_version}
	\begin{enumerate}
		\item From the above computations, Theorem \ref{thm: plur_prop1} and Corollary \ref{cor: comparison} still hold if we replace $\beta$ by $\beta + dd^c \rho$.
		
		\item We have the following:  for  measurable function $h: X \to \mathbb{R}$ which is bounded from above,
		\begin{align*}
		P_{\beta + dd^c \rho}(h) &:= \left( \sup \left\{ \varphi \in \PSH{X}{\beta + dd^c \rho} \mid \varphi \leq h \quad \text{quasi-everywhere} \right\} \right) \\
		&= \left( \sup \left\{ \widetilde{\varphi} - \rho \mid \widetilde{\varphi} \in \PSH{X}{\beta}, \widetilde{\varphi} \leq h + \rho \quad \text{quasi-everywhere} \right\} \right) \\
		&= \left( \sup \left\{ \widetilde{\varphi} \in \PSH{X}{\beta} \mid \widetilde{\varphi} \leq h + \rho \quad \text{quasi-everywhere} \right\} \right) - \rho \\
		&= P_{\beta}(h + \rho) - \rho.
		\end{align*}
		Therefore, by Remark \ref{rem: put_no_mass3}, in Theorem \ref{thm: put_no_mass4}, if we replace $\beta$ by $\beta + dd^c \rho$, the same result holds.
		
		\item (See \cite[Lemma 1.2]{GZ07} or \cite[Lemma 2.4]{LWZ24b} )Assume $\{s_j\}_j$ is a sequence converging to $+\infty$ such that $s_j \leq j$. Then $u \in \mathcal{E}(X, \beta + dd^c \rho)$ if and only if
		$$
		\lim_{j \to +\infty} (\beta + dd^c \rho + dd^c \max(u, -j))^n(\{u \leq -s_j\}) = 0.
		$$

		\item (Comparison Principle, see \cite[Proposition 3.11]{LWZ24a}). If $\varphi, \psi \in \mathcal{E}(X, \beta + dd^c \rho)$, then
		$$
		\int_{\{\varphi < \psi\}} \langle(\beta + dd^c \rho + dd^c \psi)^n \rangle\leq \int_{\{\varphi < \psi\}} \langle(\beta + dd^c \rho + dd^c \varphi)^n\rangle.
		$$
	\end{enumerate}
\end{rem}

\begin{prop} \label{prop: chac_class}
	Assume $\rho \in \PSH{X}{\beta} \cap L^{\infty}(X)$. Let $u \in \PSH{X}{\beta + dd^c \rho}$. The following are equivalent:
	\begin{enumerate}
		\item $u \in \mathcal{E}(X, \beta + dd^c \rho)$.
		\item $\env{Au}{\beta + dd^c \rho} \in \PSH{X}{\beta + dd^c \rho}$ for every $A \geq 1$.
	\end{enumerate}
\end{prop}

\begin{proof}
	The proof is mainly based on \cite[Proposition 2.5]{ALS24}.
	
	\textbf{(1) $\Rightarrow$ (2).} Assume $u \in \mathcal{E}(X, \beta + dd^c \rho)$. Define the sequence $u_j := \max(u, -j)$ and set $\varphi_j = \env{Au_j}{\beta + dd^c \rho}$. By Remark \ref{second_version}. (2), we have
	\begin{equation}\label{p1}
	\left\{
	\begin{array}{ll}
	\varphi_j &\leq Au_j \quad \text{quasi-everywhere}, \\
	(\beta + dd^c \rho + dd^c \varphi_j)^n &= \ind{\{\varphi_j = Au_j\}} (\beta + dd^c \rho + dd^c \varphi_j)^n.
	\end{array}
	\right.
	\end{equation}
	
	We have $\frac{1}{A} \varphi_j \leq u_j \quad \text{quasi-everywhere}$. Both sides of this inequality are $\beta + dd^c \rho$-psh, so $\frac{1}{A} \varphi_j \leq u_j \quad \text{everywhere}$. From \eqref{p1} and Corollary \ref{cor: comparison}, we get:
	\begin{equation} \label{p2}
	\begin{split}
	(\beta + dd^c \rho + dd^c \varphi_j)^n &\leq A^n \ind{\{\varphi_j = Au_j\}} (\beta + dd^c \rho + dd^c A^{-1} \varphi_j)^n \\
	&\leq A^n \ind{\{\varphi_j = Au_j\}} (\beta + dd^c \rho + dd^c u_j)^n.
	\end{split}
	\end{equation}
	
	Fix $k \in \mathbb{N}$. Assume by contradiction that $\varphi_j \to -\infty$ as $j \to +\infty$. By a compactness argument, it follows that $\sup_X \varphi_j \to -\infty$. Then there exists $j_0$ such that $X = \{\varphi_j \leq -Ak\}$ for $j \geq j_0$, which implies
	\begin{equation} \label{p3}
	\begin{split}
	\int_X \beta^n &\leq \int_{\{\varphi_j \leq -Ak\}} (\beta + dd^c \rho + dd^c \varphi_j)^n \\
	&\leq A^n \int_{\{u_j \leq -k\}} (\beta + dd^c (\rho + u_j))^n \\
	&= A^n \left( \int_X (\beta + dd^c (\rho + u_j))^n - \int_{\{u_j > -k\}} (\beta + dd^c (\rho + u_j))^n \right).
	\end{split}
	\end{equation}
	The second inequality follows from \eqref{p2}.
	
	For $j \geq k$, by Remark \ref{second_version}. (1), we have
	\begin{equation*}
	\begin{split}
	\int_{\{u_j > -k\}} (\beta + dd^c (\rho + u_j))^n &= \int_{\{u_j > -k\}} (\beta + dd^c \rho + dd^c \max(u_j, -k))^n \\
	&= \int_{\{u > -k\}} (\beta + dd^c (\rho + u_k))^n,
	\end{split}
	\end{equation*}
	thus by \eqref{p3},
	\begin{equation*}
	\begin{split}
	\int_X \beta^n &\leq A^n \left( \int_X (\beta + dd^c (\rho + u_j))^n - \int_{\{u > -k\}} (\beta + dd^c (\rho + u_k))^n \right) \\
	&\leq A^n \left( \int_X \beta^n - \int_{\{u > -k\}} (\beta + dd^c (\rho + u_k))^n \right).
	\end{split}
	\end{equation*}
	Letting $k \to +\infty$, by $u \in \mathcal{E}(X, \beta + dd^c \rho)$ and Definition \ref{class2}, we infer
	$$
	\int_X \beta^n \leq 0,
	$$
	which contradicts the assumption $\int_X \beta^n > 0$. Therefore, $(\varphi_j)_j$ does not converge uniformly to $-\infty$. Since $\varphi_j \downarrow \env{Au}{\beta + dd^c \rho}$, we conclude that $\env{Au}{\beta + dd^c \rho} \in \PSH{X}{\beta + dd^c \rho}$.
	
	\textbf{(2) $\Rightarrow$ (1).} Since $\int_X (\beta + dd^c \rho + dd^c u)^n \leq \int_X \beta^n$ by definition, to prove the result, we need to show the following claim:
	$$
	\int_X \langle(\beta + dd^c \rho + dd^c u)^n \rangle \geq \int_X \beta^n.
	$$
	Since $\env{Au}{\beta + dd^c \rho} \leq Au \quad \text{quasi-everywhere} \Rightarrow A^{-1} \env{Au}{\beta + dd^c \rho} \leq u \quad \text{everywhere}$, we have
	$$
	\rho + A^{-1} \env{Au}{\beta + dd^c \rho} \leq \rho + u \quad \text{everywhere}.
	$$
	By \cite[Theorem 3.3]{DDL23}, 
	\begin{equation*}
	\begin{split}
	\int_X \langle(\beta + dd^c \rho + dd^c u)^n \rangle &\geq \int_X \left \langle (\beta + dd^c \rho + dd^c A^{-1} \env{Au}{\beta + dd^c \rho}\right)^n \rangle \\
	&\geq \int_X \left(1 - \frac{1}{A}\right)^n (\beta + dd^c \rho)^n \\
	&= \left(1 - \frac{1}{A}\right)^n \int_X \beta^n.
	\end{split}
	\end{equation*}
	Letting $A \to +\infty$, the claim follows.
	
	Thus, we have shown both directions, completing the proof.
\end{proof}

\begin{cor} \label{cor: A}
	If $u \in \mathcal{E}(X, \beta + dd^c \rho)$, then $\env{Au}{\beta + dd^c \rho} \in \mathcal{E}(X, \beta + dd^c \rho)$ for $A \geq 1$.
\end{cor}

\begin{proof}
	The proof is mainly based on \cite[Corollary 2.6]{ALS24}.

	Fix $A\geq 1$. Set $v=P_{\beta+dd^c \rho}(Au) \in PSH(X, \beta+dd^c \rho)$. By Proposition \ref{prop: chac_class}, it suffices to prove $P_{\beta+dd^c \rho}(tv)\in PSH(X, \beta+dd^c \rho)$ for $t\geq 1$.
    Since $P_{\beta+dd^c \rho}(Atu) \leq t P_{\beta+dd^c \rho}(Au)$,
    we get $tv\geq P_{\beta+dd^c \rho}(Atu)\in PSH(X, \beta+dd^c \rho)$. Hence $P_{\beta+dd^c \rho}(tv) \geq P_{\beta+dd^c \rho}(Atu)\in PSH(X, \beta+dd^c \rho)$ and we conclude that $P_{\beta+dd^c \rho}(tv) \in PSH(X, \beta+dd^c \rho)$.
\end{proof}

\begin{cor} \label{cor: convex}
	$\mathcal{E}(X, \beta + dd^c \rho)$ is a convex set, and so is $\mathcal{E}(X, \beta)$.
\end{cor}

\begin{proof}
	The proof is mainly based on \cite[Proposition 1.6]{GZ07}.
	
	\textbf{Step 1:} If $\varphi \in \PSH{X}{\beta + dd^c \rho}$ and $\frac{\varphi}{2} \in \mathcal{E}(X, \beta + dd^c \rho)$, then $\varphi \in \mathcal{E}(X, \beta + dd^c \rho)$.
	
	\textit{Proof of Step 1:} Set $u = \frac{\varphi}{2}$, $u_j := \max(u, -j)$, and $\varphi_j := \max(\varphi, -j)$. Then $u_j = \frac{\varphi_{2j}}{2}$ and
	$$
	\beta + dd^c \rho + dd^c u_j = \frac{1}{2}[\beta + dd^c \rho + (\beta + dd^c \rho + dd^c \varphi_{2j})] \geq \frac{1}{2} (\beta + dd^c \rho + dd^c \varphi_{2j}).
	$$
	Thus,
	\begin{equation*}
	\begin{split}
	\int_{\{\varphi \leq -2j\}} (\beta + dd^c \rho + dd^c \varphi_{2j})^n &= \int_{\{u \leq -j\}} (\beta + dd^c \rho + dd^c \varphi_{2j})^n \\
	&\leq 2^n \int_{\{u \leq -j\}} (\beta + dd^c \rho + dd^c u_j)^n \to 0,
	\end{split}
	\end{equation*}
	which implies $\varphi \in \mathcal{E}(X, \beta + dd^c \rho)$.
	
	\textbf{Step 2:} If $\varphi \in \mathcal{E}(X, \beta + dd^c \rho)$ and $\psi \in \PSH{X}{\beta + dd^c \rho}$ such that $\varphi \leq \psi$, then $\psi \in \mathcal{E}(X, \beta + dd^c \rho)$ by \cite[Theorem 3.3]{DDL23}, and moreover, $\frac{\psi}{2} \in \mathcal{E}(X, \beta + dd^c \rho)$.
	
	
	\textbf{Step 3:} If $\varphi, \psi \in \mathcal{E}(X, \beta + dd^c \rho)$, then $\frac{\varphi + \psi}{4} \in \mathcal{E}(X, \beta + dd^c \rho)$.
	
	\textit{Proof of Step 3:} Set $w := \frac{\varphi + \psi}{4}$, $w_j := \max(w, -j)$, $\varphi_j := \max(\varphi, -j)$, and $\psi_j := \max(\psi, -j)$. Since
	$$
	\{w \leq -j\} \subseteq \{\varphi \leq -2j\} \cup \{\psi \leq -2j\},
	$$
	it suffices to prove
	$$
	(\beta + dd^c \rho + dd^c w_j)^n(\{\varphi \leq -2j\}) \to 0
	$$
	and
	$$
	(\beta + dd^c \rho + dd^c w_j)^n(\{\psi \leq -2j\}) \to 0.
	$$
	Without loss of generality, assume $\varphi, \psi \leq -2$. Then we have
	$$
	\{\varphi \leq -2j\} \subseteq \{\varphi_{2j} < w_j - j + 1\} \subseteq \{\varphi \leq -j\}.
	$$
	Hence,
	\begin{equation*}
	\begin{split}
	(\beta + dd^c \rho + dd^c w_j)^n(\{\varphi \leq -2j\}) &\leq (\beta + dd^c \rho + dd^c w_j)^n(\{\varphi_{2j} < w_j - j + 1\}) \\
	&\leq (\beta + dd^c \rho + dd^c \varphi_{2j})^n(\{\varphi \leq -j\}) \\
	&= (\beta + dd^c \rho + dd^c \varphi_j)^n(\{\varphi \leq -j\}) \to 0,
	\end{split}
	\end{equation*}
	where the second inequality follows from the comparison principle (Remark \ref{second_version} (4)), and the last equality follows from Lemma \ref{pl_prop}.
	
	By Steps 1 and 3, we conclude that $\mathcal{E}(X, \beta + dd^c \rho)$ is a convex set. The same argument applies to $\mathcal{E}(X, \beta)$ by setting $\rho = 0$.
\end{proof}

\begin{lem} \label{lem: min1}
	Assume $u, v \in \mathcal{E}(X, \beta + dd^c \rho)$. Then $\env{\min(u, v)}{\beta + dd^c \rho} \in \mathcal{E}(X, \beta + dd^c \rho)$.
\end{lem}

\begin{proof}
	The proof is mainly based on \cite[Lemma 2.9]{ALS24}.
	
	By adding a constant, we can assume without loss of generality that $u, v \leq 0$.
	
	Denote by $\varphi = \env{\min(u, v)}{\beta + dd^c \rho}$ and set $h = \frac{u + v}{2}$.
	
	Since $u, v \in \mathcal{E}(X, \beta + dd^c \rho)$, by Corollary \ref{cor: convex}, the set $\mathcal{E}(X, \beta + dd^c \rho)$ is convex. Therefore, $h = \frac{u + v}{2} \in \mathcal{E}(X, \beta + dd^c \rho)$.
	
	By Corollary \ref{cor: A}, it follows that $\env{u + v}{\beta + dd^c \rho} \in \mathcal{E}(X, \beta + dd^c \rho)$.
	
	Moreover, since $u, v \leq 0$, we have $u + v \leq \min(u, v)$. Consequently,
	$$
	\env{u + v}{\beta + dd^c \rho} \leq \env{\min(u, v)}{\beta + dd^c \rho}.
	$$
	By Step 2 in the proof of Corollary \ref{cor: convex}, if $\psi \in \mathcal{E}(X, \beta + dd^c \rho)$ and $\phi \in \PSH{X}{\beta + dd^c \rho}$ such that $\psi \leq \phi$, then $\phi \in \mathcal{E}(X, \beta + dd^c \rho)$. Applying this observation to $\psi = \env{u + v}{\beta + dd^c \rho}$ and $\phi = \env{\min(u, v)}{\beta + dd^c \rho}$, we conclude that
	$$
	\env{\min(u, v)}{\beta + dd^c \rho} \in \mathcal{E}(X, \beta + dd^c \rho).
	$$
	This completes the proof.
\end{proof}

\begin{lem} \label{lem: ineq_3}
	Assume $u, v \in \PSH{X}{\beta + dd^c \rho}$ and $\env{u, v}{\beta + dd^c \rho} \in \PSH{X}{\beta + dd^c \rho}$, and let $h \in \PSH{X}{\beta + dd^c \rho}$. Then,
	\begin{equation*}
	\begin{split}
	&\int_{X} \left| e^{P_{\beta+dd^c \rho}(u,v)} - e^h \right| \langle(\beta + dd^c \rho + dd^c P_{\beta+dd^c \rho}(u,v))^n \rangle\\
	\leq & \int_{X} \left| e^u - e^h \right| \langle (\beta + dd^c \rho + dd^c u)^n \rangle + \int_{X} \left| e^v - e^h \right| \langle(\beta + dd^c \rho + dd^c v)^n\rangle.
	\end{split}
	\end{equation*}
\end{lem}

\begin{proof}
	The proof is a direct consequence of Corollary \ref{cor: contactieq}.
\end{proof}

\begin{thm} [Domination Principle] \label{thm: domi_principal}
	Let $c \in [0, 1)$, and assume $\beta + dd^c \rho \geq 0$ with $\rho$ bounded. Let $u, v \in \mathcal{E}(X, \beta + dd^c \rho)$. If
	$$
	\langle(\beta + dd^c \rho + dd^c u)^n \rangle \leq c \langle (\beta + dd^c \rho + dd^c v)^n\rangle
	$$
	on $\{u < v\}$, then $u \geq v$.
\end{thm}

\begin{proof}
	
	The proof is mainly based on \cite{ALS24}.	
    
	By Lemma \ref{pl_prop}, we have
	$$
	c \langle(\beta + dd^c \rho + dd^c \max(u, v))^n \rangle= c \langle (\beta + dd^c \rho + dd^c v)^n \rangle  \leq \langle(\beta + dd^c \rho + dd^c v)^n\rangle
	$$
	on $\{u < v\}$. Thus, we can replace $v$ by $\max(u, v)$ and assume without loss of generality that $u \leq v$. Our goal is to prove that $u = v$, which implies the final result.
	
	For $b > 1$, define $u_b = \env{bu - (b-1)v}{\beta + dd^c \rho}$. We have $bu - (b-1)v \leq v$ for each $b > 1$ and $u_b \in \PSH{X}{\beta + dd^c \rho}$ or $u_b \equiv -\infty$.
	
	Next, we show that $u_b \in \mathcal{E}(X, \beta + dd^c \rho)$ for each $b > 1$. By Corollary \ref{cor: A}, since $u \in \mathcal{E}(X, \beta + dd^c \rho)$, we have
	\begin{equation} \label{D1}
	P_{\beta + dd^c \rho}(bu) \in \mathcal{E}(X, \beta + dd^c \rho).
	\end{equation}
	We also have
	$$
	P_{\beta + dd^c \rho}(bu) \leq bu \leq bu - (b-1)(v - \sup_X v)
	$$
	quasi-everywhere, which implies
	\begin{equation} \label{D2}
	P_{\beta + dd^c \rho}(bu) - (b-1) \sup_X v \leq bu - (b-1)v \quad \text{quasi-everywhere}.
	\end{equation}
	From \eqref{D1} and \eqref{D2}, it follows that
	$$
	u_b \geq P_{\beta + dd^c \rho}(bu) - (b-1) \sup_X v \in \mathcal{E}(X, \beta + dd^c \rho).
	$$
	By Step 2 in the proof of Corollary \ref{cor: convex}, we conclude that $u_b \in \mathcal{E}(X, \beta + dd^c \rho)$.
	
	Denote by $\mathrm{D} := \{u_b = bu - (b-1)v\}$. According to Remark \ref{second_version} (2), $(\beta + dd^c \rho + dd^c u_b)^n$ is carried by $\mathrm{D}$, and $u_b \leq bu - (b-1)v$ quasi-everywhere (i.e., $b^{-1}u_b + (1-b^{-1})v \leq u$ quasi-everywhere). Since both sides are in $\PSH{X}{\beta + dd^c \rho}$, this inequality holds everywhere. By Corollary \ref{second_version} (1),
	$$
	\ind{D} \left \langle \left( \beta + dd^c \rho + dd^c (b^{-1}u_b + (1-b^{-1})v) \right)^n \right \rangle  \leq \ind{D} \langle (\beta + dd^c \rho + dd^c u)^n \rangle.
	$$
	Hence, by hypothesis,
	\begin{align*}
	&\ind{D \cap \{u < v\}} (b^{-1})^n \langle (\beta + dd^c \rho + dd^c u_b)^n \rangle + \ind{D \cap \{u < v\}} (1-b^{-1})^n \langle (\beta + dd^c \rho + dd^c v)^n \rangle\\
	&\leq \ind{D \cap \{u < v\}} \left \langle\left( \beta + dd^c \rho + dd^c (b^{-1}u_b + (1-b^{-1})v) \right)^n \right \rangle \\
	&\leq \ind{D \cap \{u < v\}} c \langle (\beta + dd^c \rho + dd^c v)^n \rangle.
	\end{align*}
	Choose $b$ large enough such that $(1-b^{-1})^n > c$, we have
	\begin{align*}
	&\ind{D \cap \{u < v\}} (b^{-1})^n \langle (\beta + dd^c \rho + dd^c u_b)^n \rangle \\
	&\leq \ind{D \cap \{u < v\}} (c - (1-b^{-1})^n) \langle (\beta + dd^c \rho + dd^c v)^n \rangle.
	\end{align*}
	Thus, $\ind{D \cap \{u < v\}}  \langle (\beta + dd^c \rho + dd^c u_b)^n \rangle = 0$, which means $ \langle (\beta + dd^c \rho + dd^c u_b)^n \rangle$ is carried by $D \cap \{u = v\}$. On $D \cap \{u = v\}$, we have $u_b = v = u$ and $u_b \leq bu - (b-1)v \leq v$ on $X$. By Remark \ref{second_version} (1),
	\begin{equation} \label{D3}
	\ind{D \cap \{u = v\}} \langle (\beta + dd^c \rho + dd^c u_b)^n  \rangle \leq \ind{D \cap \{u = v\}} \langle (\beta + dd^c \rho + dd^c v)^n \rangle.
	\end{equation}
	
	Since $\langle (\beta + dd^c \rho + dd^c v)^n \rangle$ does not charge $\{v = -\infty\}$, one can construct an increasing function $h: \mathbb{R}^+ \to \mathbb{R}^+$ such that $h(+\infty) = +\infty$ and $h(|v|) \in L^1(\langle (\beta + dd^c \rho + dd^c v)^n \rangle )$. Thus, by \eqref{D3},
	\begin{equation} \label{D4}
	\begin{split}
	\int_X h(|u_b|) \langle (\beta + dd^c \rho + dd^c u_b)^n \rangle &= \int_{D \cap \{u = v\}} h(|v|) \langle (\beta + dd^c \rho + dd^c u_b)^n \rangle \\
	&\leq \int_X h(|v|) \langle (\beta + dd^c \rho + dd^c v)^n \rangle
    < +\infty.
	\end{split}
	\end{equation}
	
	If $(\sup_X u_b)_{b > 1}$ has a subsequence that converges to $-\infty$, then for every positive scalar $\alpha$, one can find $b$ such that $\sup_X u_b \leq -\alpha$. This implies
	\begin{equation*}
	\begin{split}
	\int_X h(|u_b|) \langle (\beta + dd^c \rho + dd^c u_b)^n  \rangle &\geq h(\alpha) \int_X \langle (\beta + dd^c \rho + dd^c u_b)^n \rangle\\
	&= h(\alpha) \int_X \beta^n,
	\end{split}
	\end{equation*}
	which is impossible because $\sup_{b \geq 1} \int_X h(|u_b|) \langle (\beta + dd^c \rho + dd^c u_b)^n \rangle < +\infty$. Therefore, $(\sup_X u_b)_{b > 1}$ is uniformly bounded.
	
	By compactness properties, $(u_b)_b$ has a subsequence $(u_{b_j})_j$ which converges in $L^1(X)$ to a function $w \in \PSH{X}{\beta + dd^c \rho}$. Fix $a > 0$. On $\{u < v - a\}$, we have $u_b \leq v - ab$, hence
	\begin{equation}
	\begin{split}
	\int_{\{u < v - a\}} w \cdot \omega^n &= \lim_{j \to +\infty} \int_{\{u < v - a\}} u_{b_j} \cdot \omega^n \\
	&\leq \lim_{j \to +\infty} (-ab_j) \int_{\{u < v - a\}} \omega^n + \int_{\{u < v - a\}} v \cdot \omega^n.
	\end{split}
	\end{equation}
	Since $w \in L^1(X)$, we infer that
	$$
	\int_{\{u < v - a\}} \omega^n = 0,
	$$
	which implies $u \geq v - a$ almost everywhere. Both sides are quasi-psh, so by subharmonicity, $u \geq v - a$ everywhere. Letting $a \to 0$, we conclude that $u = v$.
\end{proof}

\begin{cor} \label{cor: comparison1}
	Fix $\lambda > 0$. If $u_1, u_2 \in \mathcal{E}(X, \beta + dd^c \rho)$ are such that
	\begin{align*}
	e^{-\lambda u_1} \langle (\beta + dd^c \rho + dd^c u_1)^n \rangle \leq e^{-\lambda u_2} 
    \langle (\beta + dd^c \rho + dd^c u_2)^n \rangle,
	\end{align*}
	then $u_1 \geq u_2$.
\end{cor}

\begin{proof}
	
	For $a > 0$, we have 
	$$
	\langle (\beta + dd^c \rho + dd^c u_1)^n \rangle \leq e^{\lambda(u_1 - u_2)} 
    \langle (\beta + dd^c \rho + dd^c u_2)^n \rangle \leq e^{-\lambda a} \langle (\beta + dd^c \rho + dd^c u_2)^n \rangle
	$$
	on $\{u_1 < u_2 - a\}$. By Theorem \ref{thm: domi_principal}, this implies $u_1 \geq u_2$. 
	
%
%
%
%
\end{proof}

\begin{cor} \label{cor: comparison2}
	If $u, v \in \mathcal{E}(X, \beta + dd^c \rho)$ are such that
	\begin{align*}
	\langle (\beta + dd^c \rho + dd^c u)^n \rangle  \leq c \langle (\beta + dd^c \rho + dd^c v)^n \rangle 
	\end{align*}
	for some positive constant $c$, then $c \geq 1$.
\end{cor}

\begin{proof}
	
	If, in contradiction, $c < 1$. Then for each $a \in \mathbb{R}$, we have 
	$$
	\langle (\beta + dd^c \rho + dd^c u)^n  \rangle \leq c 
    \langle (\beta + dd^c \rho + dd^c (v + a))^n \rangle
	$$
	on $\{u < v + a\}$. By the Domination Principle (Theorem \ref{thm: domi_principal}), this implies $u \geq v + a$ for every $a \in \mathbb{R}$, which is a contradiction. We infer that $c \geq 1$.

\end{proof}

\section{Continuity of non-pluripolar Monge-Amp\`ere measures}
	Let $(X, \omega)$ be a Hermitian manifold and let $\beta$ be a closed real $(1, 1)$ form with $\int_X\beta^n>0$ and   there exists a bounded $\beta$-psh function $\rho$. 
In this section, we provide the proof of Theorem \ref{thm: main-1}, which corresponds to Theorem \ref{thm: maintheorem} stated below. Additionally, we prepare several results that will be used in the proof of Corollary \ref{cor: main}.

\begin{thm} [=Theorem \ref{thm: main-1}]\label{thm: maintheorem}
	Let $u_j \in \mathcal{E}(X, \beta + dd^c \rho)$ be such that $\langle (\beta + dd^c \rho + dd^c u_j)^n \rangle \leq \mu$ for some non-pluripolar Radon measure $\mu$. If $u_j \rightarrow u \in PSH(X, \beta + dd^c \rho)$ in $L^1(X)$, then $u \in \mathcal{E}(X, \beta + dd^c \rho)$ and $u_j \rightarrow u$ in capacity.
Furthermore, $\langle(\beta+dd^cu_j)^n\rangle$   converges weakly to $\langle(\beta+dd^cu)^n\rangle$.

\end{thm}

The proof of Theorem \ref{thm: maintheorem} follows the steps outlined in \cite[Theorem 3.3]{ALS24}.

\begin{thm} \label{thm: main3}
	Assume $u_j, u \in \mathcal{E}(X, \beta + dd^c \rho)$ are such that $u_j \rightarrow u$ in capacity. Assume $h_j$ is a sequence of uniformly bounded quasi-continuous functions converging in capacity to a bounded quasi-continuous function $h$. Then
	$$
	h_j \langle (\beta + dd^c \rho + dd^c u_j)^n \rangle \rightarrow h \langle (\beta + dd^c \rho + dd^c u)^n \rangle.
	$$
\end{thm}

\begin{proof}
	Let $\Theta$ be a cluster point of $(\beta + dd^c \rho + dd^c u_j)^n$ for the weak topology. We need to prove that $\Theta = (\beta + dd^c \rho + dd^c u)^n$.
	
	According to \cite[Theorem 2.6]{DDL23} \cite[Theorem 2.6]{LLZ24}, we need to show that
	$$
	\int_{X} \langle (\beta + dd^c \rho + dd^c u)^n \rangle \geq \limsup_{j \to +\infty} \int_{X} \langle (\beta + dd^c \rho + dd^c u_j)^n \rangle,
	$$
	which is evident since $u_j, u \in \mathcal{E}(X, \beta + dd^c \rho)$, and both sides of the above inequality equal $\int_{X}\beta^n$.
\end{proof}

\begin{rem}
Theorem \ref{thm: main3} states that the non-pluripolar Monge-Amp\`ere measure is continuous with respect to convergence in capacity.

	\end{rem}

\begin{lem} \label{lem: non-vanishing}
	Let $u_j \in \mathcal{E}(X, \beta + dd^c \rho)$ be such that $\langle (\beta + dd^c \rho + dd^c u_j)^n \rangle \leq \mu$ for some non-pluripolar Radon measure $\mu$. Then for any $v \in PSH(X, \beta + dd^c \rho)$, we have
	$$
	\inf_{j} \int_{X} e^{v} \langle (\beta + dd^c \rho + dd^c u_j)^n \rangle > 0.
	$$
\end{lem}

\begin{proof}
	Assume, for the sake of contradiction, that
	$$
	\inf_{j} \int_{X} e^{v} \langle (\beta + dd^c \rho + dd^c u_j)^n \rangle = 0.
	$$
	By possibly extracting a subsequence, we can assume without loss of generality that
	$$
	\lim_{j \to +\infty} \int_{X} e^v \langle (\beta + dd^c \rho + dd^c u_j)^n \rangle  = 0.
	$$
	
	Fix $c > 0$. We then have
	\begin{equation*}
	\begin{split}
	\int_{X} \beta^n &= \int_{X} \langle (\beta + dd^c \rho + dd^c u_j)^n \rangle\\
	&= \int_{\{v \leq -c\}} \langle (\beta + dd^c \rho + dd^c u_j)^n \rangle + \int_{\{v > -c\}} \langle (\beta + dd^c \rho + dd^c u_j)^n \rangle \\
	&\leq \mu(\{v \leq -c\}) + e^c \int_{X} e^v \langle (\beta + dd^c \rho + dd^c u_j)^n \rangle.
	\end{split}
	\end{equation*}
	
	First, let $j \to +\infty$, and then let $c \to +\infty$. The right-hand side approaches zero, which contradicts the fact that $\int_{X} \beta^n > 0$.
\end{proof}

\begin{lem} \label{lem: main2}
	Assume $u_j \in \mathcal{E}(X, \beta + dd^c \rho)$ satisfy the hypothesis of Theorem \ref{thm: maintheorem}, and $u_j \rightarrow u \in PSH(X, \beta + dd^c \rho)$ in $L^1(X)$. Then $u \in \mathcal{E}(X, \beta + dd^c \rho)$.
\end{lem}

\begin{proof}
	According to Proposition \ref{prop: chac_class}, it suffices to prove that $P_{\beta+dd^c \rho}(Au) \in PSH(X, \beta + dd^c \rho)$ for every $A \geq 1$. Fix $A \geq 1$ and set $\varphi_j = P_{\beta + dd^c \rho}(Au_j) \in PSH(X, \beta + dd^c \rho)$, since $u_j \in \mathcal{E}(X, \beta + dd^c \rho)$.
	
	Theorem \ref{thm: put_no_mass4} implies that $(\beta + dd^c \rho + dd^c \varphi_j)^n$ is carried by $\{\varphi_j = Au_j\}$ and that $\varphi_j \leq Au_j$.
	
	Corollary \ref{cor: comparison} implies
	\begin{equation}
	\begin{split}
	\langle(\beta + dd^c \rho + dd^c \varphi_j)^n \rangle &= 1_{\{\varphi_j = Au_j\}} \langle (\beta + dd^c \rho + dd^c \varphi_j)^n \rangle \\
	&\leq 1_{\{\varphi_j = Au_j\}} \langle \left[A(\beta + dd^c \rho) + dd^c \varphi_j \right]^n \rangle \\
	&= 1_{\{A^{-1}\varphi_j = u_j\}} A^n \left\langle \left[(\beta + dd^c \rho) + dd^c \left(\frac{1}{A}\varphi_j\right) \right]^n \right \rangle\\
	&\leq 1_{\{A^{-1}\varphi_j = u_j\}} A^n \langle \left[(\beta + dd^c \rho) + dd^c u_j \right]^n \rangle\\
	&\leq A^n \langle \left[(\beta + dd^c \rho) + dd^c u_j \right]^n \rangle \leq A^n \mu.
	\end{split}
	\end{equation}
	
	Hence, we obtain
	\begin{equation}
	\begin{split}
	\lim_{j \to +\infty} \int_{X} |e^{A^{-1}\varphi_j} - e^{u}| \langle (\beta + dd^c \rho + dd^c \varphi_j)^n \rangle  &= \lim_{j \to +\infty} \int_{\{A^{-1}\varphi_j = u_j\}} |e^{u_j} - e^{u}| \langle (\beta + dd^c \rho + dd^c \varphi_j)^n \rangle\\
	&\leq \lim_{j \to +\infty} A^n \int_{X} |e^{u_j} - e^{u}| d\mu.
	\end{split}
	\end{equation}
	
	Since $\{u_j\}_j$ is uniformly bounded from above, we can assume without loss of generality that $u_j$ are non-positive. It follows that $\{e^{u_j}\}_j$ is uniformly bounded and $e^{u_j} \in PSH(X, \beta + dd^c \rho)$. Moreover, $e^{u_j} +\rho \in PSH(X, \beta) \subseteq PSH(X, \omega)$. By \cite[Corollary 2.2]{KN22}, we have
	\begin{equation*}
	\lim_{j \to +\infty} \int_{X} \left|(e^{u_j} - \rho) - (e^{u} - \rho)\right| d\mu = 0,
	\end{equation*}
	consequently,
	\begin{equation}
	\lim_{j \to +\infty} \int_{X} |e^{A^{-1}\varphi_j} - e^{u}| \langle(\beta + dd^c \rho + dd^c \varphi_j)^n \rangle = 0.
	\end{equation}
	
	Assume, for the sake of contradiction, that $\varphi_j$ converges uniformly to $-\infty$. It follows that
	\begin{equation*}
	\lim_{j \to +\infty} \int_{X} e^{A^{-1}\varphi_j} \langle(\beta + dd^c \rho + dd^c \varphi_j)^n \rangle \leq A^n \lim_{j \to +\infty} \int_{X} e^{A^{-1} \varphi_j} d\mu = 0.
	\end{equation*}
	Therefore,
	\begin{equation*}
	\lim_{j \to +\infty} \int_{X} e^{u} \langle (\beta + dd^c \rho + dd^c \varphi_j)^n \rangle = 0,
	\end{equation*}
	which contradicts Lemma \ref{lem: non-vanishing}.
	
	By the compactness property, the sequence $(\varphi_j)_j$ has a subsequence converging to $\varphi \in PSH(X, \beta + dd^c \rho)$. Since $\varphi_j \leq Au_j$ almost everywhere, we have $\varphi \leq Au$ almost everywhere. Both sides being quasi-psh, this inequality holds everywhere. Hence, $P_{\beta + dd^c \rho}(Au) \in PSH(X, \beta + dd^c \rho)$, which implies $u \in \mathcal{E}(X, \beta + dd^c \rho)$.
\end{proof}

\begin{lem} \label{lem: main4}
	Consider $u_j, u$ as in Theorem \ref{thm: maintheorem}. Then $(u_j)_j$ has a subsequence converging to $u$ in capacity.
\end{lem}

\begin{proof}
	We may assume $u_j \leq 0$ for all $j$. By Lemma \ref{lem: main2}, we have $u \in \mathcal{E}(X, \beta + dd^c \rho)$.
	
	Up to extracting a subsequence, we may assume
	\begin{equation}
	\int_{X} |e^{u_j} - e^u| d\mu \leq 2^{-j}, \quad \forall j \geq 1.
	\end{equation}
	
	For each $k \geq j \geq 1$, consider $v_{j,k} := P_{\beta + dd^c \rho}\left( \min(u_j, \ldots, u_k) \right)$. It follows from Lemmas \ref{cor: contactieq} and \ref{lem: ineq_3} that
	\begin{equation} \label{matrix}
	\left\{
	\begin{array}{ll}
	v_{j,k} \in \mathcal{E}(X, \beta + dd^c \rho), \\
	\langle(\beta + dd^c \rho + dd^c v_{j,k})^n \rangle \leq \mu, \\
	\int_{X} |e^{v_{j,k}} - e^{u}| \langle (\beta + dd^c \rho + dd^c v_{j,k})^n \rangle \leq 2^{-j+1}.
	\end{array}
	\right.
	\end{equation}
	
	Lemma \ref{lem: non-vanishing} implies $\inf_{j,k} \int_{X} e^u \langle(\beta + dd^c \rho + dd^c v_{j,k})^n \rangle > 0$, thus by \eqref{matrix}, we have
	\begin{equation} \label{63}
	\int_{X} e^{v_{j,k}} \langle (\beta + dd^c \rho + dd^c v_{j,k})^n \rangle 
	\end{equation}
	does not converge to zero as $k \to +\infty$.
	
	If $\sup_{X} v_{j,k} \to -\infty$, then by \eqref{matrix},
	$$
	\int_{X} e^{v_{j,k}} \langle (\beta + dd^c \rho + dd^c v_{j,k})^n\rangle  \leq e^{\sup_{X} v_{j,k}} \int_{X} \beta^n \to 0
	$$
	as $k \to +\infty$, which contradicts \eqref{63}. Therefore,
	\begin{equation} \label{64}
	\sup_{X} v_{j,k}
	\end{equation}
	does not converge to $-\infty$ as $k \to +\infty$.
	
	It follows that $v_j := P_{\beta + dd^c \rho}(\inf_{k \geq j} u_k) \in PSH(X, \beta + dd^c \rho)$ since $(v_{j,k})_k$ is a decreasing sequence converging to $v_j$.
	
	By Lemma \ref{lem: main2}, $v_j \in \mathcal{E}(X, \beta + dd^c \rho)$. By construction, $(v_j)_j$ is non-decreasing. Set $v := \lim v_j = \sup v_j \in PSH(X, \beta + dd^c \rho)$. Corollary \ref{cor: convex} implies $v \in \mathcal{E}(X, \beta + dd^c \rho)$, and $v_j \to v$ in capacity by Hartogs' Lemma.
	
	Now, $|e^{v_j} - e^u|$ is uniformly bounded and quasi-continuous, and
	$$
	|e^{v_j} - e^u| - |e^v - e^u| \leq |e^{v_j} - e^v|.
	$$
	Since $e^{v_j} \to e^v$ in capacity, we have
	\begin{equation}
	|e^{v_j} - e^u| \to |e^v - e^u|
	\end{equation}
	in capacity.
	
	By Theorem \ref{thm: main3}, we have that
	\begin{equation} \label{68}
	\int_{X} |e^{v} - e^{u}| \langle(\beta + dd^c \rho + dd^c v)^n \rangle = \lim_{j \to +\infty} \int_{X} |e^{v_j} - e^{u}| \langle(\beta + dd^c \rho + dd^c v_j)^n \rangle.
	\end{equation}
	
	Similarly, by the monotone convergence theorem and \eqref{matrix},
	\begin{equation} \label{69}
	\int_{X} |e^{v_j} - e^{u}| \langle(\beta + dd^c \rho + dd^c v_j)^n\rangle = \lim_{k \to +\infty} \int_{X} |e^{v_{j,k}} - e^{u}| \langle(\beta + dd^c \rho + dd^c v_{j,k})^n \rangle = 0.
	\end{equation}
	
	Equations \eqref{68} and \eqref{69} imply that
	$$
	\int_{X} |e^{v} - e^{u}| \langle(\beta + dd^c \rho + dd^c v)^n \rangle = 0,
	$$
	thus Theorem \ref{thm: domi_principal} yields that $v \geq u$. From the definition of $v_j$, we have $v_j \leq u_j$, thus $v \leq u$. We conclude that $u = v$.
	
	Now, we have $v_j \leq u_j \leq \max(u_j, u)$, $v_j \to v = u$, and $\max(u_j, u) \to u$ in capacity. We infer that $u_j \to u$ in capacity.
\end{proof}

\begin{proof}[\textbf{End proof of Theorem \ref{thm: maintheorem}}]
	The proof follows from Theorem 3.3 in \cite{ALS24} using contradiction and applying Lemma \ref{lem: main4}.
\end{proof}

By Lemma \ref{lem: key1} and Theorem \ref{thm: key}, we have the following result.

\begin{thm}\label{thm: maintheorem2}
   Assume $\rho$ is a bounded $\beta$-psh function, and $u_j \in \mathcal{E}(X, \beta)$ such that $\langle(\beta + dd^c u_j)^n\rangle \leq \mu$ for some non-pluripolar Radon measure $\mu$. If $u_j \to u \in \operatorname{PSH}(X, \beta)$ in $L^1(X)$ and $\{u_j \}_j$ is uniformly bounded above, then $u \in \mathcal{E}(X, \beta)$ and $u_j \to u$ in capacity.
\end{thm}


To finish this section, we need to make some preparations to prove Corollary \ref{cor: main}, following the approach in \cite{KN22}.

\begin{thm}\label{thm: continuity 2}
    Let $\{u_j\}$ be a uniformly bounded sequence of $\beta$-psh functions. Assume $(\beta + dd^c u_j)^n \leq C(\beta + dd^c \varphi_j)^n$ for some uniformly bounded sequence $\{\varphi_j\} \subset \operatorname{PSH}(X, \beta)$ such that $\varphi_j \to \varphi \in \operatorname{PSH}(X, \beta)$ in capacity. If $u_j \to u \in \operatorname{PSH}(X, \beta) \cap L^{\infty}(X)$ in $L^1(X)$, then $u_j \to u$ in capacity.
\end{thm}

\begin{proof}
    This is a direct consequence of Lemma \ref{lem: cap convergence lemma2} and Theorem \ref{thm: cap convergence lemma} below.
\end{proof}

Similarly, one can obtain the following theorem.

\begin{thm}\label{thm: continuity 3}
    Let $\{u_j\}$ be a uniformly bounded sequence of $\beta$-psh functions. Assume $(\beta + dd^c u_j)^n \leq C(\omega + dd^c \varphi_j)^n$ for some uniformly bounded sequence $\{\varphi_j\} \subset \operatorname{PSH}(X, \omega)$ such that $\varphi_j \to \varphi \in \operatorname{PSH}(X, \omega)$ in capacity. If $u_j \to u \in \operatorname{PSH}(X, \beta) \cap L^{\infty}(X)$ in $L^1(X)$, then $u_j \to u$ in capacity.
\end{thm}

\begin{proof}
    This is a direct consequence of \cite[Lemma 2.3]{KN22} and Theorem \ref{thm: cap convergence lemma} below.
\end{proof}

\begin{defn}[\cite{LWZ24a}]
	
	For any Borel set $D \subset X$, the capacity of $D$ with respect to $\beta$ is defined as
	$$
	\mathrm{Cap}_{\beta}(D) := \sup \left\{ \int_{D} (\beta + dd^c h)^n : h \in PSH(X, \beta), \rho \leq h \leq \rho + 1 \right\}.
	$$
	
	The capacity of $D$ with respect to $\beta + dd^c \rho$ is defined as
	$$
	\mathrm{Cap}_{\beta + dd^c \rho}(D) := \sup \left\{ \int_{D} (\beta + dd^c \rho + dd^c h)^n : h \in PSH(X, \beta + dd^c \rho), 0 \leq h \leq 1 \right\},
	$$
	and $\mathrm{Cap}_{\beta}(D) = \mathrm{Cap}_{\beta + dd^c \rho}(D)$ by Remark 2.12 of \cite{LWZ24a}.
\end{defn}

In the sequel, let
\[
P_0 := \left\{v \in \operatorname{PSH}(X, \beta) \cap L^{\infty}(X) \mid \sup_X v = 0 \right\}
\]
be a relatively compact subset in $\operatorname{PSH}(X, \beta)$. Following the proof in \cite{KN22}, we have the following lemma:

\begin{lem}\label{lem: convergence}
    Let $\mu$ be a non-pluripolar Radon measure on $X$. Suppose $\{u_j\} \subset P_0$ converges almost everywhere with respect to the Lebesgue measure to $u \in P_0$. Then there exists a subsequence, still denoted by $\{u_j\}$, such that
    \[
    \lim_{j \to \infty} \int_X |u_j - u| \, d\mu = 0.
    \]
\end{lem}

\begin{lem}\label{lem: cap convergence lemma2}
    Let $\{u_j\}$ be as in Theorem \ref{thm: continuity 2}, and let $\{w_j\} \subset P_0$ be a uniformly bounded sequence that converges in capacity to $w \in P_0$. Then,
    \[
    \lim_{j \to \infty} \int_X |u - u_j| (\beta + dd^c w_j)^n = 0.
    \]
\end{lem}

\begin{proof}
    Observe that
    \[
    |u_j - u| = (\max\{u_j, u\} - u_j) + (\max\{u_j, u\} - u).
    \]
    Let $\phi_j := \max\{u_j, u\}$ and $v_j := \left(\sup_{k \geq j} \phi_k\right)^*$. We have $\phi_j \geq u$ and $v_j$ decreases to $u$ pointwise. By Hartogs' lemma, for any $\delta > 0$,
    \[
    \operatorname{Cap}_{\beta}\left(\{|\phi_j - u| > \delta\}\right) = \operatorname{Cap}_{\beta}\left(\{\phi_j > u + \delta\}\right) \leq \operatorname{Cap}_{\beta}\left(\{v_j > u + \delta\}\right) \to 0,
    \]
    where the last step follows since monotone convergence implies convergence in capacity.

    Next, fix $\epsilon > 0$. For large $j$, we have:
    \begin{align*}
        \int_X (\phi_j - u)(\beta + dd^c w_j)^n
        & \leq C \int_{\{|\phi_j - u| > \epsilon\}} (\beta + dd^c w_j)^n + \epsilon \int_X (\beta + dd^c w_j)^n \\
        & \leq C \operatorname{Cap}_{\beta}(|\phi_j - u| > \epsilon) + C \epsilon,
    \end{align*}
    where the first inequality follows since both $u_j$ and $\phi_j$ are uniformly bounded, and $C$ is a constant controlling them. Therefore,
    \[
    \lim_{j \to \infty} \int_X (\phi_j - u)(\beta + dd^c w_j)^n = 0.
    \]

    We next turn to the estimate of the second term. For $j > k$,
    \begin{align*}
        & \int_X (\phi_j - u_j)(\beta + dd^c w_j)^n - \int_X (\phi_j - u_j)(\beta + dd^c w_k)^n \\
        &= \int_X (\phi_j - u_j) dd^c (w_j - w_k) \wedge T \\
        &= \int_X (w_j - w_k) dd^c (\phi_j - u_j) \wedge T \\
        &\leq \int_X |w_j - w_k| \left[(\beta + dd^c \phi_j) + (\beta + dd^c u_j)\right] \wedge T \\
        &\leq C \operatorname{Cap}_{\beta}(|w_j - w_k| > \epsilon) + C \epsilon,
    \end{align*}
    where $T = \sum_{s=1}^{n-1} (\beta + dd^c w_j)^s \wedge (\beta + dd^c w_k)^{n-1-s}$ is a closed positive current. The last inequality holds because the integral $\int_X \left[(\beta + dd^c \phi_j) + (\beta + dd^c u_j)\right] \wedge T$ is dominated by $A^n \operatorname{Cap}_{\beta}(X) =A^n \int_X \beta^n$, where $A$ is the uniform bound of $\phi_j$ and $u_j$.

    Since $w_j \to w$ in capacity, the right-hand side is less than $C \epsilon$ for some uniform constant $C$, provided that $j > k > k_0$ for some large $k_0$.

    Finally, we can estimate:
    \begin{align*}
        \int_X (\phi_j - u_j)(\beta + dd^c w_j)^n &\leq \int_X (\phi_j - u_j)(\beta + dd^c w_k)^n \\
        & \quad + \left|\int_X (\phi_j - u_j)(\beta + dd^c w_j)^n - \int_X (\phi_j - u_j)(\beta + dd^c w_k)^n\right| \\
        &\leq \int_X |u - u_j| (\beta + dd^c w_j)^n + C \epsilon.
    \end{align*}
    Fixing $k = k_0$ and applying Lemma~\ref{lem: convergence}, we obtain the result by letting $j \to \infty$.
\end{proof}

\begin{thm}\label{thm: cap convergence lemma}
    Let \( u_j \in \mathrm{PSH}(X, \beta + dd^c \rho) \) be a uniformly bounded sequence converging in \( L^1(X) \) to a \( (\beta + dd^c \rho) \)-psh function \( u \). If
    \[
    \lim_{j \to +\infty} \int_X |u_j - u| \, (\beta + dd^c \rho + dd^c u_j)^n = 0,
    \]
    then \( u_j \to u \) in capacity.
\end{thm}

\begin{proof}
    Up to extracting a subsequence, we can assume that
    \[
    \int_X |u_j - u| \, (\beta + dd^c \rho + dd^c u_j)^n \leq 2^{-j}, \quad \forall j.
    \]
    Fix \( j \in \mathbb{N} \) and consider
    \[
    v_{j,k} := P_{\beta + dd^c \rho} \left(\min(u_j, \ldots, u_k)\right), \quad \text{for every } k \geq j.
    \]
    The sequence \( (v_{j,k})_k \) decreases to the function
    \[
    v_j := P_{\beta + dd^c \rho} \left(\inf_{k \geq j} u_k \right).
    \]
    By Lemma \ref{lem: ineq_3}, we have
    \[
    \int_X |v_{j,k} - u| \, (\beta + dd^c \rho + dd^c v_{j,k})^n \leq \sum_{l=j}^k \int_X |u_l - u| \, (\beta + dd^c \rho + dd^c u_l)^n.
    \]

    Since \( v_{j,k} \in \mathrm{PSH}(X, \beta + dd^c \rho) \), \( \{v_{j,k}\}_k \) is uniformly bounded, and \( v_j \) is bounded, with \( v_{j,k} \downarrow v_j \). Thus, by the Bedford-Taylor convergence theorem,
    \[
    (\beta + dd^c \rho + dd^c v_{j,k})^n \to (\beta + dd^c \rho + dd^c v_j)^n \quad \text{as } k \to +\infty.
    \]

    We now apply  \cite[Lemma 2.5]{DDL23} to show that
    \begin{equation}\label{convergence_1}
        \lim_{k \to +\infty} v_{j,k} \, (\beta + dd^c \rho + dd^c v_{j,k})^n = v_j \, (\beta + dd^c \rho + dd^c v_j)^n.
    \end{equation}

    Since \( \{u_j\}_j \) is uniformly bounded, so is \( \{v_{j,k}\}_k \). There exists \( c \geq 1 \) such that
    \[
    -c \leq v_{j,k} \leq c \quad \Rightarrow \quad 0 \leq v_{j,k} + c \leq 2c \quad \Rightarrow \quad 0 \leq \frac{1}{2c} v_{j,k} + \frac{1}{2} \leq 1.
    \]
    Therefore, there exist constants \( C, C' \leq 1 \) such that for any Borel set \( E \),
    \begin{align*}
        C \cdot (\beta + dd^c \rho + dd^c v_{j,k})^n(E)
        & \leq \left[\beta + dd^c \rho + dd^c (C v_{j,k} + C')\right]^n(E) \\
        & \leq \mathrm{Cap}_{\beta + dd^c \rho}(E) \\
        & = \mathrm{Cap}_{\beta}(E) \\
        & \leq M^n \, \mathrm{Cap}_{\omega}(E),
    \end{align*}
    where the last inequality follows from the boundedness of \( \rho \) (see \cite{LWZ24a}).

    Using a similar argument for \( v_j \), we satisfy the conditions of \cite[Lemma 2.5]{DDL23}. Since \( v_{j,k} \downarrow v_j \) and by Hartogs' lemma, \( v_{j,k} \to v_j \) in capacity, \cite[Lemma 2.5]{DDL23} implies
    \[
    \lim_{k \to +\infty} |v_{j,k} - u| \, (\beta + dd^c \rho + dd^c v_{j,k})^n = |v_j - u| \, (\beta + dd^c \rho + dd^c v_j)^n.
    \]
    Therefore,
    \begin{equation}\label{star}
        \int_X |v_j - u| \, (\beta + dd^c \rho + dd^c v_j)^n
        = \lim_{k \to +\infty} \int_X |v_{j,k} - u| \, (\beta + dd^c \rho + dd^c v_{j,k})^n
        \leq 2^{-j+1}.
    \end{equation}

    The sequence \( (v_j)_j \) is non-decreasing. By Choquet's lemma, there exists \( v \in \mathrm{PSH}(X, \beta + dd^c \rho) \) such that \( v_j \to v \) almost everywhere (and in capacity by Hartogs' lemma). Since \( (v_j)_j \) is uniformly bounded, \( v \) is bounded, thus \( v \in \mathcal{E}(X, \beta + dd^c \rho) \). By the Bedford-Taylor convergence theorem, \cite[Lemma 2.5]{DDL23}  and using \eqref{star},
    \[
    \int_X |v - u| \, (\beta + dd^c \rho + dd^c v)^n = 0.
    \]

    Since \( u_j \) converges to \( u \) in \( L^1(X) \) and is uniformly bounded, \( u \) is bounded and belongs to \( \mathcal{E}(X, \beta + dd^c \rho) \). We infer that \( v \geq u \) by Theorem \ref{thm: domi_principal}, and since \( v_j \leq u_j \) for all \( j \), it follows that \( u = v \).

    Since \( v_j \uparrow v \), \( v_j \leq u_j \), and \( v_j \to v \) in capacity, we conclude that \( u_j \to u \) in capacity.

    This completes the proof.
\end{proof}

\section{Solving Monge-Amp\`ere equations}
	Let $(X, \omega)$ be a Hermitian manifold and let $\beta$ be a closed real $(1, 1)$ form with $\int_X\beta^n>0$ and   there exists a bounded $\beta$-psh function $\rho$. In this section, we provide the proof of Corollary \ref{cor: main}. For the reader's convenience, we state it as follows.


\begin{thm}[=Corollary \ref{cor: main}]\label{thm: solve_equation_1}
	Fix \(\lambda > 0\). Let \(\mu\) be a positive Radon measure vanishing on pluripolar sets, and assume \(\mu\) is absolutely continuous with respect to  $\omega^n$, i.e., can be written as \(\mu = f \omega^n\), where \(f \in L^{1}(\omega^n)\). Then there exists a unique \(u \in \mathcal{E}(X, \beta)\) such that
	\[
	\langle(\beta + dd^c u)^n \rangle = e^{\lambda u} \mu.
	\]
\end{thm}

\begin{proof}
	The proof follows the approach of Theorem 4.2 in \cite{ALS24}, using Theorem 1.4 from \cite{LWZ24a}.
	
	\textbf{Step 1.} Assume first that \(f \in L^{\infty}(X)\). By \cite{LWZ24a}, Theorem 1.4, there exists \(u \in \operatorname{PSH}(X, \beta) \cap L^{\infty}(X)\) such that
	\[
	(\beta + dd^c u)^n = e^{\lambda u} f \omega^n.
	\]
	
	\textbf{Step 2.} For each \(j \in \mathbb{N}^*\), by Step 1, there exists \(u_j \in \operatorname{PSH}(X, \beta) \cap L^{\infty}(X)\) such that
	\[
	(\beta + dd^c u_j)^n = e^{\lambda u_j} \min(f, j) \omega^n.
	\]
	The sequence \((u_j)\) is decreasing by Corollary \ref{cor: comparison1}. Set \(u = \lim u_j\). We claim \(u \in \mathcal{E}(X, \beta)\): if \(\sup_{X} u_j \to -\infty\), then
	\[
	\int_{X} \beta^n = \int_{X} (\beta + dd^c u_j)^n \leq \int_{X} e^{\lambda \sup_{X} u_j} d\mu \to 0,
	\]
	which contradicts our assumptions. Hence, \(u \in \operatorname{PSH}(X, \beta)\), and \(u_j \to u\) in \(L^1(X)\) by the compactness properties of quasi-psh functions.
	
	Since \((\beta + dd^c u_j)^n \leq e^{\lambda u_1} d\mu\) for \(j \geq 1\), Theorem \ref{thm: maintheorem2} implies \(u \in \mathcal{E}(X, \beta)\). The convergence \((\beta + dd^c u_j)^n \to \langle(\beta + dd^c u)^n\rangle \) weakly follows from Theorem \ref{thm: main3}. Applying the Dominated Convergence Theorem, we conclude
	\[
	(\beta + dd^c u_j)^n = e^{\lambda u_j} \min(f, j) \omega^n \to e^{\lambda u} f \omega^n.
	\]
	Thus, \( \langle(\beta + dd^c u)^n\rangle = e^{\lambda u} d\mu\). Uniqueness follows from Corollary \ref{cor: comparison1}.
\end{proof}

\begin{rem}
	If we define \(\mathcal{E}^1(X, \beta)\) as
	\[
	\mathcal{E}^1(X, \beta) := \left\{u \in \mathcal{E}(X, \beta) \mid \int_X u \langle(\beta + dd^c u)^n \rangle > -\infty \right\},
	\]
	then the solution \(u\) obtained in the above theorem actually lies in \(\mathcal{E}^1(X, \beta)\). This follows from the fact that \(u e^{\lambda u}\) is bounded and both sides of the equation can be multiplied by \(u\) and integrated.
\end{rem}

\begin{rem}
	The Monge-Amp\`ere energy of \(u \in \operatorname{PSH}(X, \beta)\) is defined as
	\[
	I_{\rho}(u) := \frac{1}{n+1} \sum_{k=0}^n \int_X u \langle(\beta + dd^c u)^k\rangle \wedge (\beta + dd^c \rho)^{n-k}.
	\]
	The class \(\mathcal{E}^1(X, \beta)\) consists precisely of functions in \(\operatorname{PSH}(X, \beta)\) with finite Monge-Amp\`ere energy. Since we do not use this fact in the paper, we omit the proof.
\end{rem}

\begin{cor}
	Let \(\mu\) be a positive Radon measure vanishingand assume \(\mu\) is absolutely continuous with respect to  $\omega^n$, i.e., can be written as \(\mu = f \omega^n\), where \(f \in L^{1}(\omega^n)\). Then there exist a unique \(c > 0\) and a function \(u \in \mathcal{E}(X, \beta)\) such that
	\[
	\langle(\beta + dd^c u)^n \rangle = c \mu.
	\]
\end{cor}

\begin{proof}
	By Theorem \ref{thm: solve_equation_1}, for each \(j \geq 1\), there exists \(u_j \in \mathcal{E}(X, \beta)\) such that
	\[
	\langle(\beta + dd^c u_j)^n \rangle = e^{u_j / j} \mu.
	\]
	Setting \(v_j = u_j - \sup_{X} u_j\), we write
	\[
	\langle(\beta + dd^c v_j)^n \rangle = c_j e^{v_j / j} \mu
	\]
	for some positive number \(c_j\). By the compactness properties of quasi-psh functions, we can extract a subsequence \(v_j \to v \in \operatorname{PSH}(X, \beta)\) in \(L^1(X)\) and almost everywhere. Since \(e^{v_j / j} \mu \to \mu\) and \((c_j)\) is uniformly bounded, we obtain
	\[
	\langle(\beta + dd^c v_j)^n \rangle \to c \mu
	\]
	for some \(c > 0\). Finally, by Theorem \ref{thm: maintheorem2}, \(v \in \mathcal{E}(X, \beta)\), and Theorem \ref{thm: main3} ensures weak convergence of measures. Uniqueness of \(c\) follows from Corollary \ref{cor: comparison2}.
\end{proof}

\section{An $L^{\infty}$ a prior estimate}
	Let $(X, \omega)$ be a Hermitian manifold and let $\beta$ be a closed real $(1, 1)$ form with $\int_X\beta^n>0$ and   there exists a bounded $\beta$-psh function $\rho$.  In this section we establish the $L^{\infty}$ a prior estimates (Theorem \ref{thm: linf est}) for Monge-Ampere equations in our setting. We follow the method of \cite{GL21}. The resulting uniform bound does not depend on the bound in Skoda's uniform integrability as obtained in \cite{LWZ24a}, but rather on the $L^m$ supremum of a compact subset in $\text{PSH}(X,\beta)$. We can't make use of this fact in this paper as it is still hard to verify the uniform boundedness in order to use Theorem \ref{thm: continuity 3} when we are solving equations, but we believe that it will be useful in the future.

Following the lines of the proof of \cite{GL21}, Lemma 1.7, we obtain the following lemma:
\begin{lem}\label{lem: concave weight}
    Fix a concave increasing function $\chi:\mathbb{R}^-\rightarrow\mathbb{R}^-$ such that $\chi'(0)\geq 1$. Let $\varphi\leq0$ be a bounded $\beta_{\rho}$-psh function and $\psi=\chi\circ\varphi$. Then
    $$
(\beta_{\rho}+dd^cP_{\beta+dd^c\rho}(\psi))^n\leq1_{\{P_{\beta+dd^c\rho}(\psi)=\psi\}}(\chi'\circ\varphi)^n(\beta_{\rho}+dd^c\varphi)^n.
    $$
\end{lem}

\begin{proof}
    Set $\sigma:=\chi^{-1}:\mathbb{R}^-\rightarrow\mathbb{R}^-$, it is easy to see that $\sigma$ is a convex increasing function such that $\sigma'\leq 1$. Let $\eta:=P_{\beta+dd^c\rho}(\psi)$ and $v:=\sigma\circ\eta$.
    \begin{align*}
        \beta_{\rho}+dd^cv=&\beta_{\rho}+\sigma''\circ\eta d\eta\wedge d^c\eta+\sigma'\circ\eta dd^c\eta\\
        \geq&(1-\sigma'\circ\eta)\beta_{\rho}+\sigma'\circ\eta(\beta_{\rho}+dd^c\eta)\\
        \geq&\sigma'\circ\eta(\beta_{\rho}+dd^c\eta)\geq0.
    \end{align*}
    It follows that $v$ is $(\beta+dd^c\rho)$-psh and $(\beta_{\rho}+dd^c\eta)^n\leq1_{\{\eta=\psi\}}(\sigma'\circ\eta)^{-n}(\beta_{\rho}+dd^cv)^n$. On the contact set $\{\eta=\psi\}$ we have 
    $$
\sigma'\circ\eta=\sigma'\circ\psi=(\chi'\circ\psi)^{-1}.
    $$
    We also have $v\leq\sigma\circ\psi=\varphi$ quasi everywhere, hence everywhere on X. By corollary \ref{cor: comparison} we obtain that $(\beta_{\rho}+dd^cv)^n\leq(\beta_{\rho}+dd^c\varphi)^n$ and conclude the proof.
    
\end{proof}

\begin{thm}[=Theorem \ref{thm: linf est}]\label{thm: a prior estimate}
     Let $\mu$ be a finite positive Radon measure on X such that $\text{PSH}(X,\beta+dd^c\rho)\subset L^m(\mu)$ for some $m>n$ and $\mu(X)=\int_X\beta^n$. Then any solution $\varphi\in  \text{PSH}(X,\beta+dd^c\rho)\cap L^{\infty}(X)$ to $(\beta_{\rho}+dd^c\varphi)^n=\mu$, satisfies
     $$
Osc_X(\varphi)\leq T
     $$
     for some uniform constant T which only depends on $X,\beta$ and
     $$
A_m(\mu):=sup\left\{(\int_X(-\psi)^md\mu)^{\frac{1}{m}}: \psi\in \text{PSH}(X,\beta_{\rho})\,with\,\underset{X}{sup}\,\psi=0\right\}.
     $$ 
\end{thm}

\begin{rem}
    As noted in \cite{GL21}, this theorem applies to measures $\mu=f\omega^n$ for $0\leq f\in L^p(X)$ with $p>1$. It also applies to more general densities, such as the Orlicz weight. We refer the reader to \cite{GL21} and \cite{GL22} for more details.
\end{rem}

\begin{proof}[Proof of Theorem \ref{thm: a prior estimate}]
The proof follows from \cite{GL22} Theorem 2.2. We can assume without loss of generality that $\underset{X}{sup}\,\varphi=0$ and $\mu(X)=\int_X\beta^n=1$. Set $T_{max}:=sup\{t>0:\mu(\varphi<-t)>0\}$. Then we have $\mu(\varphi<-T_{max})=0$. It follows that $\varphi\geq T_{max}$ [$\mu$] a.e., since we have $\mu=(\beta_{\rho}+dd^c\varphi)^n$, the domination principle Theorem \ref{thm: domi_principal} implies that $\varphi\geq-T_{max}$. We only need to establish the bound for $T_{max}$.

Let $\chi:\mathbb{R}^-\rightarrow\mathbb{R}^-$ be a concave increasing weight such that $\chi(0)=0$ and $\chi'(0)=1$. Set $\psi:=\chi\circ\varphi$ and $u:=P_{\beta+dd^c\rho}(\psi)$. By lemma \ref{lem: concave weight} we have
$$
(\beta_{\rho}+dd^cu)^n\leq1_{\{u=\psi\}}(\chi'\circ\varphi)^n\mu.
$$
\textbf{Step 1.} We first control the energy of u. Fix $\epsilon>0$ small and set $m:=n+3\epsilon$. Since $\chi(0)=0$ and $\chi$ is concave, we have $|\chi(t)|\leq |t|\chi'(t)$. We can thus estimate the energy :
\begin{align*}
\int_X(-u)^{\epsilon}(\beta_{\rho}+dd^cu)^n\leq&\int_X(-\chi\circ\varphi)^{\epsilon}(\chi'\circ\varphi)^nd\mu\leq\int_X(-\varphi)^{\epsilon}(\chi'\circ\varphi)^{n+\epsilon}d\mu\\    
\leq&\left(\int_X(-\varphi)^{n+2\epsilon}d\mu\right)^{\frac{\epsilon}{n+2\epsilon}}\left(\int_X(\chi'\circ\varphi)^{n+2\epsilon}d\mu\right)^{\frac{n+\epsilon}{n+2\epsilon}}\\
\leq&A_m(\mu)^{\epsilon}\left(\int_X(\chi'\circ\varphi)^{n+2\epsilon}d\mu\right)^{\frac{n+\epsilon}{n+2\epsilon}}.
\end{align*}
Where the first inequality follows since $u=\chi\circ\varphi$ on the support of $(\beta_{\rho}+dd^cu)^n$, the second inequality follows from the fact $|\chi(t)|\leq |t|\chi'(t)$, the third follows from Holder's inequality while the last is the definition of $A_m(\mu)$.

\textbf{Step 2.} We choose a appropriate weight $\chi$ such that $\int_X(\chi'\circ\varphi)^{n+2\epsilon}d\mu\leq 2$ is under control.  We first fix a $T\in[0,T_{max})$. Recall that if $g:\mathbb{R}^+\rightarrow\mathbb{R}^+$ is an increasing, absolutely continuous function (on [0,T]) with $g(0)=1$, then
$$
\int_Xg\circ(-\varphi)d\mu=\mu(X)+\int_0^{T_{max}}g'(t)\mu(\varphi<-t)dt.
$$
Set $f(t):=\begin{cases}  
       \frac{1}{(1+t)^2\mu(\varphi<-t)}, & \text{if } t \in [0,T] \\
       \frac{1}{(1+t)^2}, & \text{if } t >T , 
       \end{cases}$ , then $f(t)$ is a $L^1$ function. Set $g(x):=\int_0^xf(t)dt+1$. It is easy to check that g is absolutely continuous on $[0,T]$ and $g'(t)=f(t)$ almost everywhere (at least at every lebesgue point). 

Using the argument again, by letting $-\chi(-x):=\int_0^xg(t)^{\frac{1}{n+2\epsilon}}dt$, then$\chi(0)=0,\chi'(0)=1$ and $g(t)=(\chi'(-t))^{n+2\epsilon}$,note that this equality holds everywhere since g is continuous and $g\geq1$. The choice guarantees that $\chi$ is concave increasing with $\chi'\geq1$, and
$$
\int_X(\chi'\circ\varphi)^{n+2\epsilon}d\mu\leq \mu(X)+\int_0^{+\infty}\frac{1}{(1+t)^2}dt\leq 2.
$$
\textbf{Step 3.} We next turn to estimate the level set $\mu(\varphi<-t)$. Using step 2, we can establish a uniform lower bound for $\underset{X}{sup}\,u$ as follows:
\begin{align*}
(-\underset{X}{sup}\,u)^{\epsilon}=&(-\underset{X}{sup}\,u)^{\epsilon}\int_X(\beta_{\rho}+dd^cu)^n\\
\leq&\int_X(-u)^{\epsilon}(\beta_{\rho}+dd^cu)^n\\
\leq&2A_m(\mu)^{\epsilon}.
\end{align*}
This implies that $0\geq\underset{X}{sup}\,u\geq -2^{\frac{1}{\epsilon}}A_m(\mu)$. We infer from this that 
$$
||u||_{L^m(\mu)}\leq(1+2^{\frac{1}{\epsilon}})A_m(\mu):=B.
$$
Indeed, there exists $c\leq2^{\frac{1}{\epsilon}}A_m(\mu)$ such that $\underset{X}{sup}\,u+c=0$, we have by definition $||u+c||_{L^m(\mu)}\leq A_m(\mu)$, and thus $||u||_{L^m(\mu)}\leq c+A_m(\mu)\leq(1+2^{\frac{1}{\epsilon}})A_m(\mu)$ by the Minkowski inequality.

From $u\leq\chi\circ\varphi\leq0$ quasi everywhere (notice that from our condition $\mu$ does not charge pluripolar set), we deduce that $||\chi\circ\varphi||_{L^m(\mu)}\leq ||u||_{L^m(\mu)}\leq B$. We thus have
\begin{align*}
    \mu(\varphi<-t)=&\int_{\{\varphi<-t\}}d\mu\leq\int_X\frac{|\chi\circ\varphi|^m}{|\chi(-t)|^m}d\mu\\
    =&\frac{1}{|\chi(-t)|^m}||\chi\circ\varphi||_{L^m(\mu)}^m\leq \frac{B^m}{|\chi(-t)|^m}.
\end{align*}
\textbf{Step 4.} Conclusion. Set $h(t):=-\chi(-t)$. We have $h(0)=0 $ and $h'(t)=[g(t)]^{\frac{1}{n+2\epsilon}}$ is positive increasing, hence h is convex. Observe also that $g(t)\geq g(0)\geq1$, we have $h'(t)\geq1$ and $h(1)=\int_0^1h'(t)dt\geq 1$. For almost all $t\in[0,T]$,
$$
\frac{1}{(1+t)^2g'(t)}=\mu(\varphi<-t)\leq \frac{B^m}{h^m(t)}.
$$
This implies that 
$$
h^m(t)\leq B^m(1+t)^2g'(t)=(n+2\epsilon)B^m(1+t)^2h''(t)(h')^{n+2\epsilon-1}(t)
$$
almost everywhere.
It follows that we can multiply both sides by $h'$ and integrate between 0 and t, we obtain
\begin{align*}
    \frac{h^{m+1}(t)}{m+1}\leq&(n+2\epsilon)B^m\int_0^th''(s)(h')^{n+2\epsilon}(s)ds\\
    \leq&\frac{(n+2\epsilon)B^m(1+t)^2}{n+2\epsilon+1}((h')^{n+2\epsilon+1}(t)-1)\\
    \leq&B^m(1+t)^2(h')^{n+2\epsilon+1}(t)
\end{align*}
Note that the first and the second inequality holds because $h$ and $h'$ are absolutely continuous and hence we can take derivative and integrate.

Recall that: $\alpha:=m+1=n+3\epsilon+1>\gamma:=n+2\epsilon+1>2$.
The previous inequality then reads $(1+t)^{-\frac{2}{\gamma}}\leq Ch'(t)h(t)^{-\frac{\alpha}{\gamma}}$ for some uniform constant C depending only on $n,m,B$. Integrating both sides between 1 and T we get
$$
\frac{(1+T)^{1-\frac{2}{\gamma}}}{1-\frac{2}{\gamma}}-\frac{2^{1-\frac{2}{\gamma}}}{1-\frac{2}{\gamma}}\leq C\int_1^T\frac{(h^{1-\frac{\alpha}{\gamma}})'}{1-\frac{\alpha}{\gamma}}\leq C_1h(1)^{1-\frac{\alpha}{\gamma}}\leq C_1
$$
since $1-\frac{\alpha}{\gamma}<0$ and $h(1)\geq 1$. From this we deduce that T can be controlled by a constant $C_2$ which depends only on $A_m(\mu)$.

Finally, letting $T\rightarrow T_{max}$ we conclude the proof.
\end{proof}

\end{document}